\documentclass[a4paper,oneside,11pt]{article}
\pdfoutput=1

\usepackage[left=3.5cm,right=3.5cm,top=2.5cm,bottom=3cm]{geometry}
\usepackage{bm,amsfonts,amssymb,amsmath,amsthm}
\numberwithin{equation}{section}
\usepackage{graphicx,float,epstopdf}
\usepackage{enumerate}
\usepackage{xcolor}
\usepackage{hyperref}

\newtheorem{theorem}{Theorem}[section]
\newtheorem{definition}[theorem]{Definition}
\newtheorem{lemma}[theorem]{Lemma}
\newtheorem{corollary}[theorem]{Corollary}

\newtheorem{remark}[theorem]{Remark}

\theoremstyle{definition}
\newtheorem{algorithm}{Algorithm}[section]

\begin{document}

\title{Normalized Wolfe-Powell-type local minimax method for finding multiple unstable solutions of nonlinear elliptic PDEs\thanks{This work was supported by NSFC grants 12171148, 11771138 and the Construct Program of the Key Discipline in Hunan Province. Liu's work was also partially supported by the NSFC grants 12101252 and 11971007. Yi's work was also partially supported by the NSFC grant 11901185, National Key R\&D Program of China (No. 2021YFA1001300) and the Fundamental Research Funds for the Central Universities 531118010207.}}

\author{Wei Liu\thanks{Key Laboratory of Computing and Stochastic Mathematics (Ministry of Education), Hunan Normal University, Changsha, Hunan 410081, China. \emph{Present address:} South China Research Center for Applied Mathematics and Interdisciplinary Studies, South China Normal University, Guangzhou 510631, China (Email: {\tt wliu@m.scnu.edu.cn}).} 
\and 
Ziqing Xie\thanks{Key Laboratory of Computing and Stochastic Mathematics (Ministry of Education), Hunan Normal University, Changsha, Hunan 410081, China (Email: {\tt ziqingxie@hunnu.edu.cn}).}
\and 
Wenfan Yi\thanks{Corresponding author. School of Mathematics, Hunan University, Changsha, Hunan 410082, China (Email: {\tt wfyi@hnu.edu.cn}).}
}

\date{}

\maketitle

\vspace{-3ex}

\begin{abstract}\footnotesize
The local minimax method (LMM) proposed in [Y. Li and J. Zhou, SIAM J. Sci. Comput., 23(3), 840--865 (2001)] and [Y. Li and J. Zhou, SIAM J. Sci. Comput., 24(3), 865--885 (2002)] is an efficient method to solve nonlinear elliptic partial differential equations (PDEs) with certain variational structures for multiple solutions. The steepest descent direction and the Armijo-type step-size search rules are adopted in [Y. Li and J. Zhou, SIAM J. Sci. Comput., 24(3), 865--885 (2002)] and play a significant role in the performance and convergence analysis of traditional LMMs. In this paper, a new algorithm framework of the LMMs is established based on general descent directions and two normalized (strong) Wolfe-Powell-type step-size search rules. The corresponding algorithm framework named as the normalized Wolfe-Powell-type LMM (NWP-LMM) is introduced with its feasibility and global convergence rigorously justified for general descent directions. As a special case, the global convergence of the NWP-LMM algorithm combined with the preconditioned steepest descent (PSD) directions is also verified. Consequently, it extends the framework of traditional LMMs. In addition, conjugate gradient-type (CG-type) descent directions are utilized to speed up the NWP-LMM algorithm. Finally, extensive numerical results for several semilinear elliptic PDEs are reported to profile their multiple unstable solutions and compared for different algorithms in the LMM's family to indicate the effectiveness and robustness of our algorithms. In practice, the NWP-LMM combined with the CG-type direction indeed performs much better than its known LMM companions. 

\vspace{1ex}
{\bfseries Key words.} 
semilinear elliptic PDEs, 
multiple unstable solutions,
local minimax method,
normalized strong Wolfe-Powell-type search rule,
conjugate gradient-type descent direction,
general descent directions,
global convergence

\vspace{1ex}
{\bfseries AMS subject classifications.}
35J20, 35B38, 65N12, 65J15, 65Jxx
\end{abstract}

% 35J20 % Variational methods for second-order elliptic equations
% 35B38 % Critical points of functionals in context of PDEs (e.g., energy functionals)
% 65N12 % Stability and convergence of numerical methods for boundary value problems involving PDEs
% 65J15 % Numerical solutions to equations with nonlinear operators
% 65Jxx % Numerical analysis in abstract spaces

\section{Introduction}

Various nonlinear problems in physics, chemistry, biology and materials sciences can be reduced to consider multiple solutions of the Euler-Lagrange equation associated with a continuously Fr\'{e}chet-differentiable nonlinear functional $E$ defined on a real Hilbert space $H$, i.e.,
\begin{equation}\label{eq:ELeq}
  E'(u)=0, \quad u\in H,
\end{equation}
where $E'$ is the Fr\'{e}chet-derivative of $E$. Actually, solutions of the Euler-Lagrange equation \eqref{eq:ELeq} are called critical points of the functional $E$. The most well-known candidates for critical points are local extrema to which classical variational and optimization methods have contributed a lot. 

Nowadays, with the development of new experimental techniques, it has been possible to observe local unstable equilibria or transient excited states in numerous physical/chemical/biological systems. Consequently, their theoretical and numerical studies have attracted increasing attentions \cite{Chang1993,CZN2000IJBC,LWZ2013JSC,Rabinowitz1986,XYZ2012SISC,ZRSD2016npj}. However, these local unstable equilibria or transient excited states are related to critical points that are not local extrema, and then called saddle points. Virtually, in terms of the instability analysis, for a critical point $u_*$ with its second-order Fr\'{e}chet-derivative $E''(u_*)$ existing, its instability can be depicted by its Morse index (MI) \cite{Chang1993}, which is defined as the maximal dimension of subspaces of $H$ on which the linear operator $E''(u_*)$ is negative-definite. In fact, for a nondegenerate critical point $u_*$, i.e., $E''(u_*)$ is invertible, if its $\mathrm{MI}=0$, it is a strict local minimizer and then a stable critical point, while if its $\mathrm{MI}>0$, then it is an unstable critical point. Generally speaking, the higher MI means the more instability.

Owing to the instability and multiplicity of saddle points, the design and analysis of numerical methods for grasping saddle points in a stable way is much more challenging than that for stable critical points. In recent years, various numerical methods were developed to capture saddle points in a stable way, such as the (climbing) string method \cite{ERV2002PRB,RV2013JCP}, the gentlest ascent dynamics \cite{EZ2011NL} and the (shrinking) dimer method \cite{HJ1999JCP,ZD2012SINUM}. Nevertheless, the methods mentioned above mainly focus on finding saddle points with $\mathrm{MI}=1$.

In this paper, we are interested in stable and efficient numerical computations for multiple saddle points with high MIs. Existing methods in this area include the search extension method \cite{CX2004CMA,CX2008SCM} and its modified versions \cite{LXY2021SSM,XCX2005IMA}, the augmented partial Newton method and its variants \cite{LWZ2017JSC,XYZ2015JCAM}, the high-index optimization-based shrinking dimer method \cite{YZZ2019SISC} and its extension to non-gradient systems \cite{YYZ2021SCM}, etc.. Recently, some dynamical methods for finding constrained saddle points with high MIs were also developed \cite{LXY-CGAD,YHZ2022JSC}. On the other hand, motivated by classical minimax theorems in the critical point theory (see, e.g., \cite{Rabinowitz1986} and references therein) and numerical researches of Choi-McKenna \cite{CM1993NATMA}, Ding-Costa-Chen \cite{DCC1999NA} and Chen-Zhou-Ni \cite{CZN2000IJBC}, Li and Zhou proposed a local minimax method (LMM) for various high-MI saddle points based on a local minimax characterization of them \cite{LZ2001SISC}. Then in \cite{XYZ2012SISC}, Xie et al. modified the LMM with a significant relaxation on the domain of the {\em local peak selection}, which is a crucial notion for the LMM and will be illustrated in details later. According to \cite{LZ2001SISC,XYZ2012SISC}, the LMM grasps a saddle point with $\mathrm{MI}=n$ ($n\in\mathbb{N}^+$) by dealing with a two-level local minimax problem as
\begin{equation}\label{eq:minimax}
  \min_{v\in S_H}\max_{w\in[L, v]}E(w),
\end{equation}
where $S_H=\{v\in H: \|v\|=1\}$ is the unit sphere with $\|\cdot\|$ the norm in $H$, $L\subset H$ is a given $(n-1)$-dimensional closed subspace usually constructed based on some known or previously found critical points, and $[L, v]=\{tv+w_L:t\geq0,w_L\in L\}$ denotes a closed half subspace. Actually, the inner local maximization is an optimization problem in the $n$-dimensional half subspace $[L,v]$, which can be solved efficiently by standard optimization algorithms in Euclidean spaces. The outer constrained local minimization, which is generally infinite-dimensional and much more challenging in the numerical computation, is the major concern of the LMM. For the sake of handling this task, the LMM adopts a normalized iterative scheme (NIS) with the steepest descent direction $d^{SD}_k=-g_k$, i.e., 
\begin{equation}\label{eq:lmmiter}
  v_{k+1}=v_k(\alpha_k)=\frac{v_k-\alpha_k g_k}{\|v_k-\alpha_k g_k\|},\quad
  w_{k+1}=p(v_{k+1}),\quad k= 0, 1, 2,\ldots,
\end{equation}
where $\alpha_k>0$ is a step-size and $g_k=\nabla E(w_k)\in H$ is the gradient of $E$ at $w_k=p(v_k)$ with $p(v_k)$ representing a local maximizer of $E$ on $[L, v_k]$, known as the so-called {\em local peak selection} of $E$ w.r.t. $L$ at $v_k$. Actually, the local peak selection $p(v_k)$ can be expressed as $p(v_k)=t_k v_k+w_k^L$ for some $t_k\geq0$ and $w_k^L\in L$ \cite{LZ2001SISC}.

One of fundamental problems for the NIS \eqref{eq:lmmiter} is how to choose a suitable step-size $\alpha_k$. In earliest implementations of the LMM, a normalized exact step-size search rule was employed and aimed to find the step-size $\alpha_k>0$ such that the functional $E(p(v_k(\alpha)))$ attains its minimum \cite{LZ2001SISC}, i.e.,
\begin{equation}\label{ex-s}
  E(p(v_k(\alpha_k))) = \min_{\alpha>0} E(p(v_k(\alpha))),\quad k= 0, 1, 2,\ldots.
\end{equation}
Nevertheless, such a step-size search rule is very expensive in practical computations and it is even hard to establish the global convergence of the corresponding LMM algorithm. To compensate for this shortage, several normalized inexact step-size search rules with the advantages of low computational cost and easy implementation have been introduced in the literature to choose the step-size $\alpha_k>0$ such that the decrease amount $E(p(v_k))-E(p(v_k(\alpha_k)))>0$ is acceptable.

Note that a widely applied normalized inexact step-size search rule in traditional LMMs is the normalized Armijo-type step-size search rule, which was first introduced in \cite{LZ2002SISC} and further simplified in \cite{YZ2005SISC} as the following form 
\begin{equation}\label{eq:armijo}
  E(p(v_k(\alpha_k))) \leq E(p(v_k))-\frac{1}{4} \alpha_k t_k\|g_k\|^2,\quad k= 0, 1, 2,\ldots.
\end{equation}
In fact, the factor ${1}/{4}$ can be replaced by any constant $\sigma\in(0,1)$ \cite{LXY2021CMS}. Thanks to this step-size search rule, global convergence results for the normalized Armijo-type LMM (NA-LMM) algorithm were established in \cite{LZ2002SISC,Z2017CAMC}. However, the decrease condition \eqref{eq:armijo} is satisfied for all sufficiently small step-sizes (see Fig.~\ref{fig:SearchRulesOld}), hence some artificial safeguards are needed to prevent step-sizes from being too small and the algorithm from interminable backtracking \cite{LZ2002SISC,Z2017CAMC}. Actually, the backtracking strategy chooses the largest step-size in the sequence $\{\lambda\rho^m\}_{m\in\mathbb{N}}$ (for a given trial step-size $\lambda>0$ and a backtracking factor $\rho\in(0,1)$) that satisfies the normalized Armijo-type search rule. It plays an important role not only in the numerical implementation but also in the convergence analysis of the NA-LMM. Nevertheless, a choice of appropriate parameters $\lambda$ and $\rho$ is not known a priori. Recently, a normalized Goldstein-type step-size search rule was proposed in \cite{LXY2021CMS} to guarantee the sufficient decrease of the functional and prevent step-sizes from being too small simultaneously. The feasibility and global convergence analysis of the normalized Goldstein-type LMM (NG-LMM) were also provided in \cite{LXY2021CMS}. Actually, the normalized Goldstein-type step-size search rule in it makes progress with two inequalities, which can be formulated as
 \begin{equation}\label{eq:goldstein}
  -\delta \alpha_k t_k\|g_k\|^2 \leq E(p(v_k(\alpha_k)))-E(p(v_k)) \leq -\sigma \alpha_k t_k\|g_k\|^2,\quad k= 0, 1, 2,\ldots,
\end{equation}
with constants $\sigma$ and $\delta$ satisfying $0<\sigma<\delta<1$. Unfortunately, as shown in Fig.~\ref{fig:SearchRulesOld}, the normalized Goldstein-type step-size search rule may exclude the minimizer $\alpha_*$ of $E(p(v_k(\alpha)))$ outside the acceptable interval $[\bar{\alpha}_1,\bar{\alpha}_2]$. Thus, more effective and reasonable step-size search rules may devote to the LMM for credibly capturing saddle points of the functional $E$.

\begin{figure}[!ht]
\centering
\includegraphics[width=.6\textwidth]{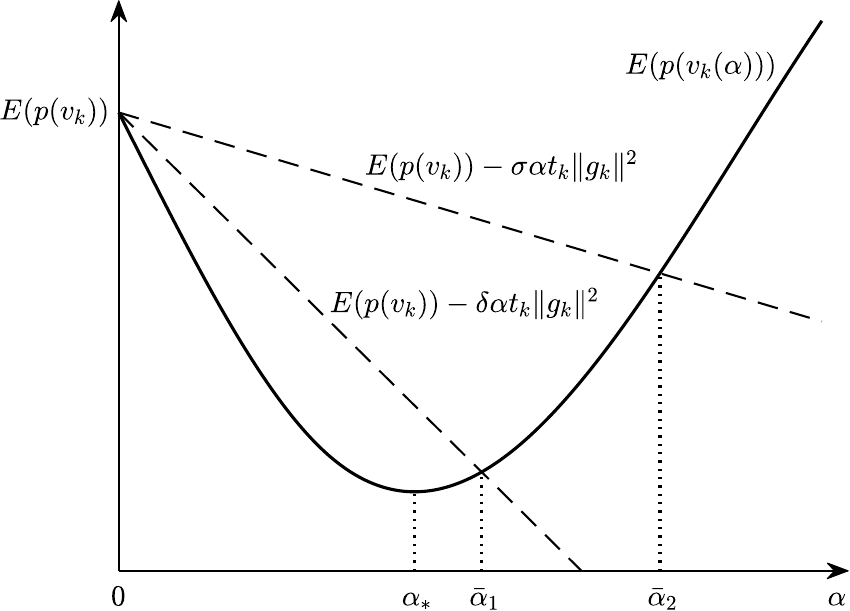}
\caption{Illustration of normalized Armijo- and Goldstein-type step-size search rules in traditional LMMs: the acceptable interval of the  Armijo-type step-size is $(0,\bar{\alpha}_2]$, while the acceptable interval of the Goldstein-type step-size is $[\bar{\alpha}_1,\bar{\alpha}_2]$.}
\label{fig:SearchRulesOld}
\end{figure}

Furthermore, the convergence rate is another fundamental problem of the NIS \eqref{eq:lmmiter} for solving the outer minimization in the two-level local optimization problem \eqref{eq:minimax}. As shown in the NIS \eqref{eq:lmmiter}, the steepest descent direction $d_k^{SD}=-g_k=-\nabla E(p(v_k))$ was chosen as a descent direction in all existing LMM algorithms since 2001. As a result, they have certain limitations in terms of convergence rate. It is well known that, in line search algorithms for unconstrained optimizations in Euclidean spaces, there are many choices of descent directions that may have better performance than the steepest descent direction, such as the conjugate gradient (CG) direction and the quasi-Newton direction, which usually have rapid convergence rate. Therefore, in this paper we try to design some improved iterative schemes of the form as \eqref{eq:lmmiter} by replacing the steepest descent direction $d_k^{SD}=-g_k$ by other more efficient descent directions $d_k\in H$ for the outer minimization process to improve the numerical performance and convergence rate of the LMMs' family. We note that the CG and quasi-Newton methods in optimization theory are often used in combination with some Wolfe-Powell line search strategy \cite{Powell1976,W1969SIAMRev,W1971SIAMRev} that consists of the Armijo condition and a curvature condition. In fact, such a curvature condition on the step-size is particularly important for both algorithm implementation and convergence analysis of the CG and quasi-Newton methods.

Inspired by the discussions above, this paper is aimed to develop a new LMM framework based on general descent directions and the (strong) Wolfe-Powell-type step-size search rules, called a normalized Wolfe-Powell-type LMM (NWP-LMM), to capture multiple unstable solutions of the Euler-Lagrange equation \eqref{eq:ELeq} and provide the possibility to speed up the convergence. By employing some curvature properties, two types of normalized Wolfe-Powell-type step-size search rules will be introduced for the LMM with general descent directions. Their mathematical justifications and global convergence will be established rigorously for general descent directions by making full use of these curvature properties, which obviously distinguish from those of the NA-LMM \cite{LZ2002SISC,XYZ2012SISC} and NG-LMM \cite{LXY2021CMS}. Finally, two types of descent directions, i.e., the preconditioned steepest descent (PSD) direction and the CG-type descent direction, will be proposed and compared to be implemented in our NWP-LMM algorithm for computing multiple unstable solutions of several semilinear elliptic partial differential equations (PDEs), such as the nonlinear Schr\"{o}dinger equation (NLSE), H\'{e}non equation and Chandrasekhar equation. Indeed, it will be seen that the CG-type descent direction can greatly speed up our NWP-LMM algorithm. It is worthwhile to point out that the steepest descent direction can be replaced by a general descent direction in the devise of the traditional normalized Armijo-type and Goldstein-type LMM algorithms. Further, both the feasibility and global convergence of them can be verified in the line of our approach.

The rest of this paper is organized as follows. Firstly, some preliminaries for the LMM are provided in section~\ref{sec:pre}. Then, in section~\ref{sec:nw-lmm}, the NWP-LMM framework based on general descent directions and the normalized (strong) Wolfe-Powell-type step-size search rules is introduced. Its feasibility and some related properties are also discussed in this section. Global convergence of the NWP-LMM algorithm with general descent directions is verified rigorously in section~\ref{sec:gcvg}. In addition, two different types of descent directions, i.e., the PSD and CG-type descent directions, are proposed and analyzed in section~\ref{sec:dk} to feasibly implement our NWP-LMM algorithms. Furthermore, section~\ref{sec:numer} reports the detailed numerical results in 2D including the numerical comparison of different LMM algorithms for above mentioned semilinear elliptic PDEs to illustrate the effectiveness and robustness of our approach. Finally, some conclusions are drawn in section~\ref{sec:cons}.

\section{Preliminaries}\label{sec:pre}

For the convenience of discussions later, we introduce some notations and basic lemmas for the LMM in this section.

Let $(\cdot,\cdot)$ and $\|\cdot\|$ be respectively the inner product and norm in the Hilbert space $H$. Denote $2^H$ as the set of all subsets of the Hilbert space $H$, $S_H=\{v\in H:\|v\|=1\}$ as the unit sphere in $H$ and $X^\bot$ as the orthogonal complement to a subspace $X\subset H$. Suppose that $L$, serving as a so-called support space later, is a given closed finite-dimensional subspace in $H$. Define the half subspace $[L, v]:=\{tv+w^L:t\geq0,w^L\in L\}$ for any $v\in S_H$. Throughout this paper, we assume that the functional $E$ has a local minimizer at $0\in H$ and focus on finding nontrivial saddle points of $E$. The {\em local peak selection}, a crucial notion, is defined as follows (cf. \cite{LZ2001SISC,LXY2021CMS,XYZ2012SISC}).
\begin{definition}\label{def:pv}
The {\em peak mapping} of $E$ w.r.t. $L$ is a set-valued mapping $P:S_H\to2^H$ s.t., for any $v\in S_H$, $P(v)$ is the set of all local maximum points of $E$ on $[L,v]$. 
A {\em peak selection} of $E$ w.r.t. $L$ is a single-valued mapping $p:S_H\to H$ s.t.
\[ p(v)\in P(v),\quad \forall \, v\in S_H. \]
For a given $v\in S_H$, we say that $E$ has a {\em local peak selection} w.r.t. $L$ at $v$ if there is a neighborhood $\mathcal{N}(v)$ of $v$ and a mapping $p:\mathcal{N}(v)\cap S_H\to H$ s.t.
\[ p(u)\in P(u),\quad \forall\, u\in \mathcal{N}(v)\cap S_H. \]
\end{definition}

Let $\langle\cdot,\cdot\rangle$ denote the duality pairing between $H$ and its dual space $H^*$. By definition, each local peak selection $p(v)$, with $v\in S_H$, can be expressed as $p(v)=t_vv+w_v^L$ with $t_v\geq0$ and $w_v^L\in L$. To avoid the degeneracy, we always assume that $t_v>0$, i.e., $p(v)\notin L$, in the subsequent analysis. The above definition implies that $p(v)$ belongs to the well-known Nehari manifold $\mathcal{N}_E:=\{u\in H\backslash\{0\}:\langle E'(u),u\rangle=0\}$, which contains all nontrivial critical points of the functional $E$. In fact, the following orthogonality holds obviously.

% Let $\langle\cdot,\cdot\rangle$ denote the duality pairing between $H$ and its dual space $H^*$. The above definition implies that each local peak selection $p(v)$, with $v\in S_H$, belongs to the well-known Nehari manifold $\mathcal{N}_E:=\{u\in H\backslash\{0\}:\langle E'(u),u\rangle=0\}$, which contains all nontrivial critical points of the functional $E$. In fact, the following orthogonality holds obviously.

\begin{lemma}[\cite{LXY2021CMS,XYZ2012SISC}]\label{lem:orth}
Assume that $E\in C^1(H,\mathbb{R})$ has a local peak selection $p$ w.r.t. $L$ at $v\in S_H$ satisfying $p(v)\notin L$. Then, $\langle E'(p(v)),w\rangle=0$, $\forall w\in[L, v]$. In particular, $\langle E'(p(v)),p(v)\rangle=0$.
\end{lemma}

The following property follows from direct computation and is frequently utilized in the feasibility and convergence discussions for the LMM-type algorithms.

\begin{lemma}[\cite{LXY2021CMS}]\label{lem:pv-tv}
Let $\bar{v}\in S_H\backslash L$. Suppose that the local peak selection $p$ of $E$ w.r.t. $L$ is continuous at $\bar{v}$. Denote $p(v)=t_v v+w_v^L$ and $p(\bar{v})=t_{\bar{v}} \bar{v}+w_{\bar{v}}^L$, where $t_v,t_{\bar{v}}\geq0$ and $w_v^L,w_{\bar{v}}^L\in L$. If $v \rightarrow \bar{v}$, then $t_v\rightarrow t_{\bar{v}}$ and $w_v^L\to w_{\bar{v}}^L$.
\end{lemma}

Imitating similar lines of the proof for Theorem~2.1 in \cite{LZ2001SISC}, the following local minimax principle can be obtained and is referred to Theorem~4.1 in \cite{LXY2021CMS}.
\begin{theorem}\label{thm:lmt0}
If $E\in C^1(H,\mathbb{R})$ has a local peak selection w.r.t. $L$ at $\bar{v}\in S_H\backslash L$, denoted by $p(\bar{v})=t_{\bar{v}}\bar{v}+w_{\bar{v}}^L$, satisfying (i) $p$ is continuous at $\bar{v}$; (ii) $t_{\bar{v}}>0$; and (iii) $\bar{v}$ is a local minimizer of $E(p(v))$ on $S_H$, then $p(\bar{v})\notin L$ is a critical point of $E$.
\end{theorem}

Since the local peak selection $p(\bar{v})$ is a local maximizer of $E$ on the half subspace $[L,\bar{v}]$, Theorem~\ref{thm:lmt0} characterizes a saddle point of $E$ as a solution to the local minimax problem \eqref{eq:minimax}, or equivalently, the local minimization problem of $E$ on the {\em solution submanifold}
\[ \mathcal{M}=\{p(v):\, v\in S_H\}. \]
When a saddle point $u^*$ of $E$ (known as an unstable critical point in $H$) can be characterized as a local solution to the two-level optimization problem \eqref{eq:minimax} of the form $u^*=p(v^*)$ with $v^*$ minimizing $E(p(v))$ on $S_H$, it becomes stable on $\mathcal{M}$, i.e.,
\begin{equation}\label{eq:minEonM}
E(u^*)=\min_{u\in\mathcal{M}} E(u).
\end{equation}
We remark here that, under certain conditions, similar to Theorem~2.2 in \cite{LZ2001SISC} and Theorem~1.5 in \cite{LZ2002SISC}, the existence of the local minimizer of $E(p(v))$ on a subspace of $S_H$ can be verified by employing the Ekeland's variational principle. 
In addition, Theorem~\ref{thm:lmt0} also indicates an important feature of the LMM, which shows that it can stably find different saddle points and avoid computing those we have found. In fact, in order to obtain multiple saddle points, the LMM needs to repeatedly solve the local minimization problem \eqref{eq:minEonM} or the two-level local minimax problem \eqref{eq:minimax} for different choices of $L$, which is usually spanned by some found critical points. Under assumptions of Theorem~\ref{thm:lmt0}, a local minimizer of $E$ on $\mathcal{M}$ is a critical point different from those in $L$. As a result, Theorem~\ref{thm:lmt0} provides a mathematical justification that the LMM can find unstable saddle points of the functional $E$ in a stable way.

Virtually, suitable descent algorithms can work for the minimization process in the optimization problem \eqref{eq:minimax} or \eqref{eq:minEonM}. As noted above, in traditional LMMs, the steepest descent direction serves as a search direction to numerically solve the local minimax problem \eqref{eq:minimax}. While, in this paper, the NWP-LMM will be constructed and analyzed for general descent directions. 
% It is worthwhile to point out that the classical NA-LMM \cite{LZ2002SISC,XYZ2012SISC} and NG-LMM \cite{LXY2021CMS} and their mathematical justifications can also be extended to the case of general descent directions by similar process.

\section{Normalized Wolfe-Powell-type LMM}\label{sec:nw-lmm}
In this section, a NWP-LMM framework for general descent directions will be proposed to capture multiple saddle points of the functional $E$ via solving the optimization problem \eqref{eq:minimax}. In order to hit this goal, two types of normalized Wolfe-Powell-type step-size search rules with general descent directions will be introduced and analyzed by involving curvature conditions. We adopt the same notations as those in the section~\ref{sec:pre} unless specified and begin with some essential properties.

\begin{figure}[!ht]
\centering
\vspace{1ex}
\includegraphics[width=.5\textwidth]{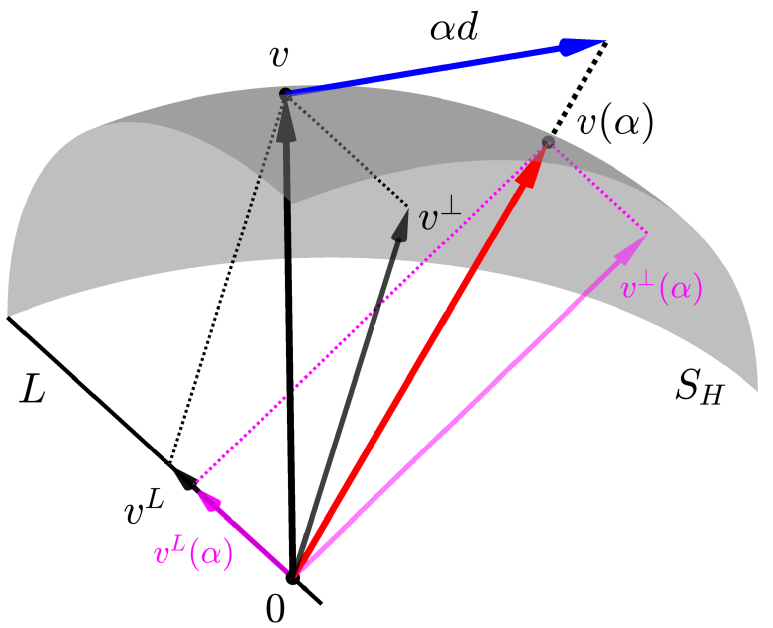}
\caption{Illustration of the normalized iterative scheme \eqref{eq:vs-expr}.}
\label{fig:NIS}
\end{figure}

Let $v\in S_H\backslash L$ and $d\in[L,v]^\bot$. Write $v=v^L+v^\bot$ with $v^L\in L$ and $v^\bot\in L^\bot\backslash\{0\}$. For all $\alpha\in\mathbb{R}$, there holds the following orthogonal decomposition (see Fig.~\ref{fig:NIS}),
\begin{equation}\label{eq:vs-expr}
  v(\alpha)=\frac{v+\alpha d}{\|v+\alpha d\|}
  = \frac{v+\alpha d}{\sqrt{1+\alpha^2\|d\|^2}}
  = v^L(\alpha)+v^\bot(\alpha)\in S_H,
\end{equation}
where
\[
  v^L(\alpha)=\frac{v^L}{\sqrt{1+\alpha^2\|d\|^2}}\in L,\quad
  v^\bot(\alpha)=\frac{v^\bot+\alpha d}{\sqrt{1+\alpha^2\|d\|^2}}\in L^\bot.
\]
Obviously, for all $\alpha\in\mathbb{R}$, $\|v^L(\alpha)\|\leq\|v^L\|$. Since $\|v(\alpha)\|^2=\|v^L(\alpha)\|^2+\|v^\bot(\alpha)\|^2=1$, it follows that $\|v^\bot(\alpha)\|\geq\|v^\bot\|>0$, $\forall\alpha\in\mathbb{R}$. Consequently, $v(\alpha)\in S_H\backslash L$, $\forall\alpha\in\mathbb{R}$. Further, we have the following property.

\begin{lemma}\label{lem:vs}
Let $v\in S_H\backslash L$, $d\in[L,v]^\bot$ and $v(\alpha)$ be expressed as in \eqref{eq:vs-expr}. Then
\begin{equation}\label{eq:vssd}
  \|v(\alpha+s)-v(\alpha)\|\leq|s|\|d\|, \quad \forall\alpha,s\in\mathbb{R}.
\end{equation}
\end{lemma}
\begin{proof}
Denote $\ell_1=\sqrt{1+(\alpha+s)^2\|d\|^2}$ and $\ell_2=\sqrt{1+\alpha^2\|d\|^2}$. 
Noting that $\ell_1,\ell_2\geq1$, % and $\|v(\alpha+s)-v(\alpha)\|^2=2-2(v(\alpha+s),v(\alpha))$, 
we have
\begin{align*}
s^2\|d\|^2
 &= \|\ell_1v(\alpha+s)-\ell_2v(\alpha)\|^2 \\
 &= \ell_1^2+\ell_2^2-2\ell_1\ell_2(v(\alpha+s),v(\alpha)) \\
 &= (\ell_1-\ell_2)^2+\ell_1\ell_2\|v(\alpha+s)-v(\alpha)\|^2 \\
 &\geq \|v(\alpha+s)-v(\alpha)\|^2. 
 \qedhere
\end{align*}
\end{proof}

Set $p$ to be a local peak selection of $E$ w.r.t. $L$ at $v\in S_H\backslash L$. 
The following lemma is crucial in constructing the normalized Wolfe-Powell-type step-size search rule.

\begin{lemma}\label{lem:dEpvs}
Suppose $E\in C^1(H,\mathbb{R})$ and let $v\in S_H\backslash L$, $d\in[L,v]^\bot$, $v(\alpha)$ be expressed as in \eqref{eq:vs-expr} and $p(v(\alpha))=t_{v(\alpha)}v(\alpha)+w^L_{v(\alpha)}$ be a local peak selection of $E$ w.r.t. $L$ at $v(\alpha)$. 
% Denote $p(v(\alpha))=t_{v(\alpha)}v(\alpha)+w^L_{v(\alpha)}$ with $t_{v(\alpha)}\geq0$ and $w^L_{v(\alpha)}\in L$. 
If $p$ is locally Lipschitz continuous around $v(\alpha)$ and $t_{v(\alpha)}>0$, then the composite function $\alpha\mapsto E(p(v(\alpha)))$ is continuously differentiable and
\begin{equation}\label{eq:dEpvs}
  \frac{\mathrm{d}}{\mathrm{d} \alpha}E(p(v(\alpha))) = \hat{t}_{v(\alpha)}\langle E'(p(v(\alpha))),d\rangle,
\end{equation}
where $\hat{t}_{v(\alpha)}=t_{v(\alpha)}/\sqrt{1+\alpha^2\|d\|^2}$.
\end{lemma}
\begin{proof}
The mean value theorem states that, when $s\in\mathbb{R}$ is close to zero,
\begin{align}
& E(p(v(\alpha+s)))-E(p(v(\alpha)))-\langle E'(p(v(\alpha))),p(v(\alpha+s))-p(v(\alpha))\rangle \nonumber\\
  &\qquad\qquad\qquad= \langle E'(\xi)-E'(p(v(\alpha))),p(v(\alpha+s))-p(v(\alpha))\rangle \nonumber\\
  &\qquad\qquad\qquad\leq \|E'(\xi)-E'(p(v(\alpha)))\|_{H^*}\|p(v(\alpha+s))-p(v(\alpha))\|, \label{eq:djpvs-eq1}
\end{align}
where $\xi = p(v(\alpha))+\theta(p(v(\alpha+s))-p(v(\alpha)))$ for some $\theta=\theta(\alpha)\in(0,1)$. By the continuity of $E'$, $p$ and $v(\alpha)$, we have 
\[ \|E'(\xi)-E'(p(v(\alpha)))\|_{H^*}\to0\quad \mbox{as} \quad s\to0. \] 
The local Lipschitz continuity of $p$ and Lemma~\ref{lem:vs} imply that the right-hand side of \eqref{eq:djpvs-eq1} is $o(\|p(v(\alpha+s))-p(v(\alpha))\|)=o(\|v(\alpha+s)-v(\alpha)\|)=o(s\|d\|)$, and therefore, by Lemma~\ref{lem:orth},
\begin{align*}
E(p(v(\alpha+s)))-E(p(v(\alpha)))
  &= \langle E'(p(v(\alpha))), p(v(\alpha+s))-p(v(\alpha))\rangle + o(s\|d\|) \\
  &= t_{v(\alpha+s)} \langle E'(p(v(\alpha))),v(\alpha+s)\rangle + o(s\|d\|) \\
  &= \frac{t_{v(\alpha+s)}s}{\sqrt{1+(\alpha+s)^2\|d\|^2}} \langle E'(p(v(\alpha))),d\rangle + o(s\|d\|).
\end{align*}
From Lemma~\ref{lem:pv-tv}, we obtain $t_{v(\alpha+s)}\to t_{v(\alpha)}$ as $s\to0$. It follows that
\[
 \lim_{s\to0}\frac{E(p(v(\alpha+s)))-E(p(v(\alpha)))}{s} 
 = \frac{t_{v(\alpha)}}{\sqrt{1+\alpha^2\|d\|^2}} \langle E'(p(v(\alpha))),d\rangle,\quad \alpha\in\mathbb{R}. 
\]
In other words, $E(p(v(\alpha)))$ is differentiable w.r.t. $\alpha$ and \eqref{eq:dEpvs} holds.
Finally, it is clear that the right-hand side of \eqref{eq:dEpvs} is continuous w.r.t. $\alpha$. The proof is completed.
\end{proof}

\begin{remark}
If $E$ possesses a higher regularity, say, $E\in C^1(H,\mathbb{R})$ and $E':H\to H^*$ is locally $\gamma_1$-H\"{o}lder continuous for some $\gamma_1\in(0,1]$, then the regularity assumption of $p$ in Lemma~\ref{lem:dEpvs} can be relaxed to that $p$ is locally $\gamma_2$-H\"{o}lder continuous around $v(\alpha)$ for some $\gamma_2\in(1/(\gamma_1+1),1]$. Actually, the key step in the proof is to justify that the right-hand side of \eqref{eq:djpvs-eq1} is $o(s\|d\|)$. The local H\"{o}lder continuity of $E'$ and $p$ implies that there exist two constants $C_1(\alpha),C_2(\alpha)>0$ s.t., when $s\to0$,
\begin{align*}
  &\|E'(\xi)-E'(p(v(\alpha)))\|_{H^*}\|p(v(\alpha+s))-p(v(\alpha))\| \\
  &\qquad\qquad\qquad\leq C_1(\alpha)\|\xi-p(v(\alpha))\|^{\gamma_1}\|p(v(\alpha+s))-p(v(\alpha))\| \\
  &\qquad\qquad\qquad= C_1(\alpha)\,\theta^{\gamma_1}\|p(v(\alpha+s))-p(v(\alpha))\|^{\gamma_1+1} \\
  &\qquad\qquad\qquad\leq C_1(\alpha)C_2(\alpha)\|v(\alpha+s)-v(\alpha)\|^{\gamma_2(\gamma_1+1)},
\end{align*}
where $\xi = p(v(\alpha))+\theta(p(v(\alpha+s))-p(v(\alpha)))$ and $\theta=\theta(\alpha)\in(0,1)$ are the same as in \eqref{eq:djpvs-eq1}. Consequently, the right-hand side of \eqref{eq:djpvs-eq1} is $o(s\|d\|)$.
\end{remark}

\subsection{Normalized Wolfe-Powell-type step-size search rules}
Now, we are ready to propose normalized Wolfe-Powell-type step-size search rules for general descent directions by constructing curvature conditions to not only prevent step-sizes from being too small, but also avoid excluding the local minimizer of the function $E(p(v(\alpha)))$ w.r.t. $\alpha$. In other words, the optimal step-size is always contained in the feasible step-size interval.

\begin{definition}[Descent direction]\label{def:descdir}
Let $p$ be a local peak selection w.r.t. $L$ at $v\in S_H\backslash L$. Assume that $E\in C^1(H,\mathbb{R})$ and $E'(p(v))\neq0$. A vector $d\in[L,v]^\bot$ is called a {\em descent direction} of $E$ w.r.t. $L$ at $p(v)$ if $\langle E'(p(v)),d\rangle<0$.
\end{definition}

\begin{remark}
It is noted that the above definition is slightly different from the usual definition of descent direction in optimization theory. In fact, the descent direction $d$ in Definition~\ref{def:descdir} is not only required to decrease the functional $E$ at $p(v)$ (i.e., $\langle E'(p(v)),d\rangle<0$) but also to satisfy the orthogonality condition $d\,\bot\,[L,v]$. We make some comments on the reasons for introducing the latter condition as follows.
\begin{enumerate}[(i)]
\item The descent direction $d$ in Definition~\ref{def:descdir} is for the outer-level minimization and should be relatively independent of the inner-level maximization. Intuitively, an iteration along a descent direction in $[L,v]^\bot$, i.e., a descent direction with zero component in $[L,v]$, tends to 
% preserve the maximality of $p(v)$ on $[L,v]$ and 
enhance the stability of the algorithm.
\item By Lemma~\ref{lem:orth}, the condition $d\in[L,v]^\bot$ is automatically satisfied for the steepest descent direction $d^{SD}=-\nabla E(p(v))$ at $p(v)$, i.e., the Riesz representer of $-E'(p(v))$ determined by $(d^{SD},\phi) = -\langle E'(p(v)),\phi\rangle$, $\forall \phi\in H$. Actually, this orthogonality plays a very important role in both algorithm implementation and theoretical analysis for the classical LMMs. Based on this observation, preserving this orthogonality to the general descent direction $d$ for the outer-level minimization is a preferred choice.
\end{enumerate}
\end{remark}

\begin{definition}[Normalized Wolfe-Powell-type step-size search rules]\label{def:wolfe}
Suppose $E\in C^1(H,\mathbb{R})$ and let $v\in S_H\backslash L$, $p$ be a peak selection of $E$ w.r.t. $L$, and $d\in[L, v]^\bot$ be a descent direction of $E$ w.r.t. $L$ at $p(v)$. Denote $p(u)=t_u u+w^L_u$ ($u\in S_H$) with $t_u\geq0$ and $w^L_u\in L$. For two given constants $\sigma_1$ and $\sigma_2$ with $0<\sigma_1<\sigma_2<1$, we say that the step-size $\alpha>0$ satisfies
\begin{itemize}
\item[$\bullet$] the {\em normalized Wolfe-Powell-type step-size search rule} at $v$, if there hold
\begin{subequations}\label{eq:wolfe}
\begin{align}
  &E(p(v(\alpha))) \leq E(p(v))+\sigma_1\alpha t_v\langle E'(p(v)),d\rangle, \label{eq:wolfe-1} \\
  &\hat{t}_{v(\alpha)}\langle E'(p(v(\alpha))),d\rangle \geq \sigma_2t_v\langle E'(p(v)),d\rangle, \label{eq:wolfe-2}
\end{align}
\end{subequations}
where $\hat{t}_{v(\alpha)}=t_{v(\alpha)}/\sqrt{1+\alpha^2\|d\|^2}$; 
\item[$\bullet$] the {\em normalized strong Wolfe-Powell-type step-size search rule} at $v$, if there hold
\begin{subequations}\label{eq:swolfe}
\begin{align}
  &E(p(v(\alpha))) \leq E(p(v))+\sigma_1\alpha t_v\langle E'(p(v)),d\rangle, \label{eq:swolfe-1} \\
  &\hat{t}_{v(\alpha)}\left|\langle E'(p(v(\alpha))),d\rangle\right| \leq -\sigma_2t_v\langle E'(p(v)),d\rangle.\label{eq:swolfe-2}
\end{align}
\end{subequations}
\end{itemize}
\end{definition}

The condition \eqref{eq:wolfe-1} or \eqref{eq:swolfe-1} is referred to as the sufficient decrease condition, while conditions \eqref{eq:wolfe-2} and \eqref{eq:swolfe-2} are referred to as curvature conditions. It is pointed out that, if the steepest descent direction is employed, \eqref{eq:wolfe-1} (or \eqref{eq:swolfe-1}) is equivalent to the normalized Armijo-type condition used in traditional LMMs \cite{LZ2002SISC,XYZ2012SISC,YZ2005SISC}.

The feasibility of normalized (strong) Wolfe-Powell-type step-size search rules above is provided as follows.

\begin{theorem}\label{thm:feasib}
Let $E\in C^1(H,\mathbb{R})$, $v\in S_H\backslash L$, $d\in[L,v]^\bot$ and $p$ be a peak selection of $E$ w.r.t. $L$. Denote $p(u)=t_uu+w^L_u$, $u\in S_H$, with $t_u\geq0$ and $w^L_u\in L$. Assume that (i) $p$ is locally Lipschitz continuous on the curve $\{v(\alpha):\alpha\geq0\}$; (ii) $t_v>0$; (iii) $d$ is a descent direction of $E$ w.r.t. $L$ at $p(v)$, i.e., $\langle E'(p(v)),d\rangle<0$; and (iv) $\inf_{\alpha>0}E(p(v(\alpha)))>-\infty$. Then, for given $\sigma_1$, $\sigma_2$ with $0<\sigma_1<\sigma_2<1$, there exist two positive constants $\bar{\alpha}_1$, $\bar{\alpha}_2$ with $\bar{\alpha}_1<\bar{\alpha}_2$ s.t., for any $\alpha\in(\bar{\alpha}_1,\bar{\alpha}_2)$, it satisfies the normalized strong Wolfe-Powell-type step-size search rule \eqref{eq:swolfe} and therefore the normalized Wolfe-Powell-type step-size search rule \eqref{eq:wolfe}.
\end{theorem}
\begin{proof}
Since \eqref{eq:swolfe} yields \eqref{eq:wolfe}, we only need to verify that there exists an interval $(\bar{\alpha}_1,\bar{\alpha}_2)$ s.t. \eqref{eq:swolfe} holds for all $\alpha\in(\bar{\alpha}_1,\bar{\alpha}_2)$. Set $\varphi(\alpha):=E(p(v(\alpha)))$, $\alpha\geq0$ and $\varphi(0)=E(p(v))$. In view of Lemma~\ref{lem:pv-tv} and Lemma~\ref{lem:dEpvs}, we have $\varphi\in C^1([0,\infty),\mathbb{R})$ and $\varphi'(\alpha)=\hat{t}_{v(\alpha)}\langle E'(p(v(\alpha))),d\rangle$ with $\hat{t}_{v(\alpha)}=t_{v(\alpha)}/\sqrt{1+\alpha^2\|d\|^2}$. Note that conditions (ii) and (iii) imply $\varphi'(0)=t_v\langle E'(p(v)),d\rangle<0$. Hence, \eqref{eq:swolfe} can be rewritten as
\begin{equation}\label{eq:wp-in-phi}
  \varphi(\alpha) \leq \varphi(0) + \sigma_1\alpha\varphi'(0),\quad
  |\varphi'(\alpha)| \leq -\sigma_2\varphi'(0),\quad \alpha\geq0,
\end{equation}
and, for all $\alpha>0$ small enough, $\varphi(\alpha)<\varphi(0)+\sigma_1\alpha\varphi'(0)$ holds. In addition, since $\varphi(0)+\sigma_1\alpha\varphi'(0)\to-\infty$ as $\alpha\to+\infty$ and the condition (iv) states that $\varphi(\alpha)$ is bounded from below for all $\alpha>0$, apparently $\varphi(\alpha)>\varphi(0)+\sigma_1\alpha\varphi'(0)$ holds for all $\alpha>0$ large enough. Consequently, the equation
\begin{equation}\label{eq:phialeq}
  \varphi(\alpha)=\varphi(0)+\sigma_1\alpha\varphi'(0), \alpha>0,
\end{equation}
admits at least one positive solution. Let $\bar{\alpha}>0$ be the smallest positive solution to the equation \eqref{eq:phialeq}. Then, it implies that
\begin{equation}\label{eq:phi-ineq1}
  \varphi(\alpha)<\varphi(0)+\sigma_1\alpha\varphi'(0),\quad \forall \alpha\in(0,\bar{\alpha}).
\end{equation}
From the mean value theorem, there exists $\tilde{\alpha}\in(0,\bar{\alpha})$ s.t. $\varphi(\bar{\alpha})-\varphi(0)=\varphi'(\tilde{\alpha})\bar{\alpha}$, which leads to $\varphi'(\tilde{\alpha})=\sigma_1\varphi'(0)$ by the definition of $\bar{\alpha}$. The facts that $0<\sigma_1<\sigma_2<1$ and $\varphi'(0)<0$ imply that $\sigma_2\varphi'(0)<\sigma_1\varphi'(0)=\varphi'(\tilde{\alpha})<0$. Therefore, $|\varphi'(\tilde{\alpha})|<-\sigma_2\varphi'(0)$. By the continuity of $\varphi'(\alpha)$, there exists $0<\delta<\min\{\tilde{\alpha},\bar{\alpha}-\tilde{\alpha}\}$ s.t.
\begin{equation}\label{eq:phi-ineq2}
  |\varphi'(\alpha)|<-\sigma_2\varphi'(0),\quad \forall \alpha\in(\tilde{\alpha}-\delta,\tilde{\alpha}+\delta)\subset(0,\bar{\alpha}).
\end{equation}
Setting $\bar{\alpha}_1=\tilde{\alpha}-\delta$ and $\bar{\alpha}_2=\tilde{\alpha}+\delta$, the combination of \eqref{eq:phi-ineq1}-\eqref{eq:phi-ineq2} states that \eqref{eq:wp-in-phi} holds for all $\alpha\in(\bar{\alpha}_1,\bar{\alpha}_2)$. %The proof is finished.
\end{proof}

\begin{remark}
Fig.~\ref{fig:strongWolfeRule} provides a geometric interpretation of the feasibility of the normalized (strong) Wolfe-Powell-type step-size search rule. 
\end{remark}

\begin{figure}[!ht]
\centering
\includegraphics[width=.65\textwidth]{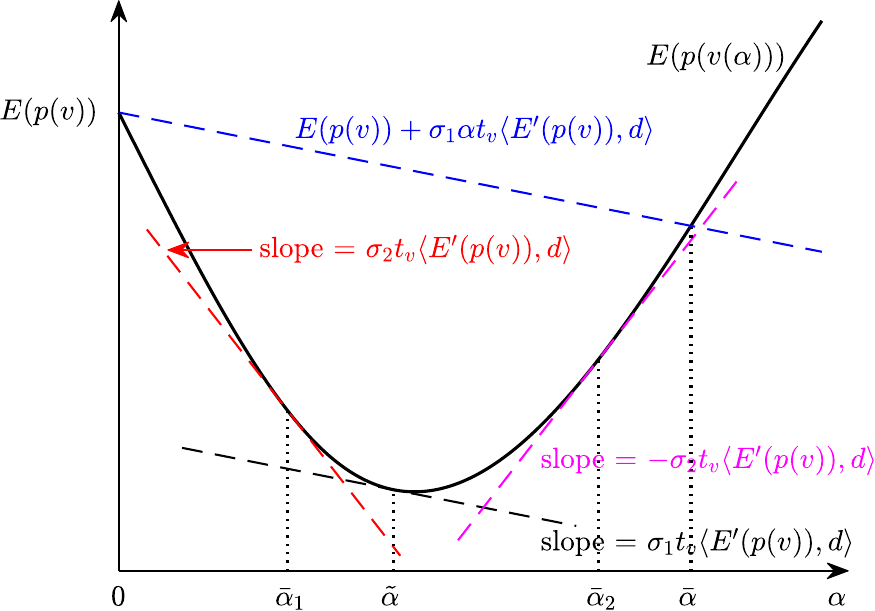}
\caption{Illustration of the normalized (strong) Wolfe-Powell-type step-size search rule: the acceptable interval of the Wolfe-Powell-type step-size is $[\bar{\alpha}_1, \bar{\alpha}]$, while the acceptable interval of the strong Wolfe-Powell-type step-size is $[\bar{\alpha}_1, \bar{\alpha}_2]$. Here, only the case $\bar{\alpha}_2\leq\bar{\alpha}$ is shown in the figure; for the case $\bar{\alpha}_2>\bar{\alpha}$, we redefine $\bar{\alpha}_2=\bar{\alpha}$.}
\label{fig:strongWolfeRule}
\end{figure}

\subsection{Normalized Wolfe-Powell-type local minimax algorithm}

Following the idea of traditional LMMs \cite{LZ2001SISC,LZ2002SISC,XYZ2012SISC}, the vital steps of the NWP-LMM algorithm are described in Algorithm~\ref{alg:wplmm}.

\begin{algorithm}\label{alg:wplmm}
{\rm\bfseries Normalized Wolfe-Powell-type Local Minimax Algorithm.}
\begin{enumerate}[\bf Step 1.]
  \item Take constants $\varepsilon>0$, $\sigma_1,\sigma_2$ with $0<\sigma_1<\sigma_2<1$, and $n-1$ previously found critical points $u_1,u_2,\ldots,u_{n-1}$ of $E$ where $u_{n-1}$ is the one with  the highest critical value in $\{u_i\}$ ($1\leq i\leq n-1$). Set $L={\rm span}\{u_1,u_2,\ldots,u_{n-1}\}$, let $k:=0$, and choose an initial ascent direction $v_0=v_0^L+v_0^\bot\in S_H$ at $u_{n-1}$ with $v_0^L\in L$, $v_0^\bot\in L^\bot$ and $v_0^\bot\neq0$. With an initial guess $w=v_0+u_{n-1}$, solve for
      \[ w_0=\arg\max_{w\in[L,v_0]}E(w), \]
      and denote $w_0=p(v_0)=t_0v_0+w_0^L$, where $t_0\geq0$ and $w_0^L\in L$.
  \item Compute a descent direction $d_k\in[L, v_k]^\bot$ of $E$ w.r.t. $L$ at $w_k=p(v_k)$ s.t. $\langle E'(w_k),d_k\rangle<0$, which will be discussed in section~\ref{sec:dk}.
  \item If the {\em stopping criterion} $\|E'(w_k)\|_{H^*}<\varepsilon$ is satisfied (or more criteria are satisfied if necessary), then output $u_n=w_k$ and stop; otherwise, go to {\bf Step~4}.
  \item Set $v_k(\alpha) =\frac{v_k+\alpha d_k}{\|v_k+\alpha d_k\|}$ and find a step-size $\alpha_k>0$ satisfying the normalized Wolfe-Powell-type step-size search rule, i.e.,  
  \begin{subequations}\label{eq:wp-k}
  \begin{align}
    &E(p(v_k(\alpha_k))) \leq E(w_k)+\sigma_1\alpha_kt_k\langle E'(w_k),d_k\rangle, \label{eq:wp-k-1}  \\
    &\hat{t}_k(\alpha_k)\langle E'(p(v_k(\alpha_k))),d_k\rangle \geq \sigma_2t_k\langle E'(w_k),d_k\rangle, \label{eq:wp-k-2}
  \end{align}
  \end{subequations}
  or the normalized strong Wolfe-Powell-type step-size search rule, i.e.,
  \begin{subequations}\label{eq:swp-k}
  \begin{align}
    &E(p(v_k(\alpha_k))) \leq E(w_k)+\sigma_1\alpha_kt_k\langle E'(w_k),d_k\rangle, \label{eq:swp-k-1}  \\
    &\hat{t}_k(\alpha_k)\big|\langle E'(p(v_k(\alpha_k))),d_k\rangle\big| \leq -\sigma_2t_k\langle E'(w_k),d_k\rangle, \label{eq:swp-k-2}
  \end{align}
  \end{subequations}
  where $p(v_k(\alpha_k))=t_k(\alpha_k)v_k(\alpha_k)+w_k^L(\alpha_k)$ with $t_k(\alpha_k)\geq0$ and $w_k^L(\alpha_k)\in L$ is the local maximizer of $E$ on $[L, v_k(\alpha_k)]$ computed by utilizing $w=t_k v_k(\alpha_k)+w_k^L$ as an initial guess and $\hat{t}_k(\alpha_k)=t_k(\alpha_k)/\sqrt{1+\alpha_k^2\|d_k\|^2}$.
  \item Set $v_{k+1}=v_k(\alpha_k)$, $t_{k+1}=t_k(\alpha_k)$, $w_{k+1}^L=w_k^L(\alpha_k)$ and $w_{k+1}=p(v_{k+1})=t_{k+1}v_{k+1}+w_{k+1}^L$. Then, update $k:=k+1$ and go to {\bf Step~2}.
\end{enumerate}
\end{algorithm}

We remark here that, similar to the classical (strong) Wolfe-Powell line search algorithm in optimization theory (see, e.g., \cite{Fletcher1987,Nocedal2006,SunYuan2006}), one can employ an interpolation approach to efficiently find a step-size $\alpha_k>0$ satisfying the normalized Wolfe-Powell-type step-size search rule \eqref{eq:wp-k} or the normalized strong Wolfe-Powell-type step-size search rule \eqref{eq:swp-k} in Step~4 of Algorithm~\ref{alg:wplmm}. The implementation details are skipped here for brevity.

\section{Global convergence}\label{sec:gcvg}
In this section, we establish the global convergence of Algorithm~\ref{alg:wplmm}. In order to hit this goal, the following concept of compactness is needed. We simply utilize the same notations as those in Algorithm~\ref{alg:wplmm} throughout this section.
\begin{definition}[\cite{Rabinowitz1986}]
A functional $E\in C^1(H,\mathbb{R})$ is said to satisfy the Palais-Smale (PS) condition if every sequence $\{w_n\}\subset H$ s.t. $\{E(w_n)\}$ is bounded and $E'(w_n)\to0$ in $H^*$ has a convergent subsequence.
\end{definition}

It is pointed out that the following lemma gives a significant behavior of the sequence generated by Algorithm~\ref{alg:wplmm}, which does not depend on the choice of descent directions and step-sizes in the algorithm. The proof is similar to that of Lemma~2.3 in \cite{XYZ2012SISC} and omitted here for brevity.
\begin{lemma}\label{lem:mono}
Let $\{v_k\}$ be a sequence generated by Algorithm~\ref{alg:wplmm} with $v_0\in S\backslash L$. Denote $v_k=v_k^\bot+v_k^L$ with $v_k^\bot\in L^\bot$ and $v_k^L\in L$, $k=0,1,\ldots$,  then there hold $\|v_0^\bot\|\leq\|v_k^\bot\|\leq1$ and $v_k^L=\tau_kv_0^L$ with $0<\tau_{k+1}\leq\tau_k\leq1$ for  $k=0,1,\ldots$.
\end{lemma}

Lemma~\ref{lem:mono} states that once an initial ascent direction $v_0=v_0^L+v_0^\bot\in S\backslash L$ is chosen in Algorithm~\ref{alg:wplmm}, the closed subset 
\begin{equation}\label{eq:V0}
  \mathcal{V}_0:=\{v=v^\bot+\tau v_0^L\in S_H:v^\bot\in L^\bot,0\leq\tau\leq1\}\subset S_{[L^\bot,v_0^L]}\subset S_H\backslash L
\end{equation}
contains all possible vectors $v_k$ that Algorithm~\ref{alg:wplmm} may generate. Thus, the domain of a peak selection can be limited to be $\mathcal{V}_0$ instead of $S_H$.

We remark here that, in general, the local peak selection $p(v)$ defined on $\mathcal{V}_0$ is no longer a homeomorphism. The following weak version related to the homeomorphism property of the local peak selection $p(v)$ plays a significant role for establishing the global convergence. It can be verified by an analogous argument to that of Theorem~2.1 in \cite{XYZ2012SISC} and only the continuity of the peak selection $p(v)$ on $\mathcal{V}_0$ is sufficient. Consequently, we skip the proof for simplicity.
\begin{lemma}\label{lem:home}
 Suppose $E\in C^1(H,\mathbb{R})$ and let $p$ be a peak selection of $E$ w.r.t. $L$ satisfying {\rm(i)} $p$ is continuous on $\mathcal{V}_0$ and {\rm(ii)} $t_k\geq\delta$ for some $\delta>0$, $\forall\,k=0,1,\ldots$. If the sequence $\{w_k\}$ generated by Algorithm~\ref{alg:wplmm} contains a subsequence $\{w_{k_i}\}$ converging to some $u_*\in H$, then the corresponding subsequence $\{v_{k_i}\}$ converges to some $v_*\in\mathcal{V}_0$ with $u_*=p(v_*)$.
\end{lemma}

Before proving the global convergence of Algorithm~\ref{alg:wplmm}, we give some assumptions on the general descent direction, which are quite reasonable and hold for many descent directions, especially for the steepest descent direction used in traditional LMMs.
\begin{enumerate}[\rm({A}1)]
  \item $\langle E'(w_k),d_k\rangle\leq -c_1\|E'(w_k)\|_{H^*}^2$ for some $c_1>0$, $\forall\, k=0,1,\ldots$;
  \item $\|d_k\|\leq c_2\|E'(w_k)\|_{H^*}$ for some $c_2>0$, $\forall\, k=0,1,\ldots$; 
  \item if $\{v_k\}$ contains a subsequence $\{v_{k_i}\}$ converging to some $\bar{v}$ with $E'(p(\bar{v}))\neq0$, then the corresponding descent direction subsequence $\{d_{k_i}\}$ converges.
\end{enumerate}
Here, the assumption (A1) serves as a strong descent condition, and the assumption (A2) requires the length of the descent direction $d_k$ to be controlled by that of the gradient. In addition, the assumption (A3) admits a certain weak version of the continuous dependency of the descent direction $d_k$ on $v_k$.

The global convergence of Algorithm~\ref{alg:wplmm} is as follows.

\begin{theorem}\label{thm:cvg-wplmm}
Suppose $E\in C^1(H,\mathbb{R})$ and let $p$ be a peak selection of $E$ w.r.t. $L$, $\{v_k\}$ and $\{w_k\}$ be sequences generated by Algorithm~\ref{alg:wplmm}. If 
{\rm(i)} $p$ is locally Lipschitz continuous on $\mathcal{V}_0$;
{\rm(ii)} $t_k\geq\delta$ for some $\delta>0$, $\forall\, k=0,1,\ldots$; 
{\rm(iii)} $\inf_{v\in\mathcal{V}_0}E(p(v))>-\infty$ and assumptions {\rm(A1)}-{\rm(A3)} hold, then
\begin{enumerate}[\rm(a)]
  \item $\sum\limits_{k=0}^{\infty}\alpha_k\|E'(w_k)\|_{H^*}^2<\infty$;
  \item every accumulation point of $\{w_k\}$ is a critical point not belonging to $L$; 
  \item $\liminf\limits_{k\to\infty}\|E'(w_k)\|_{H^*}=0$.
\end{enumerate}
Further, if the functional $E$ satisfies the (PS) condition, then
\begin{enumerate}[\rm(a)]
\setcounter{enumi}{3}
  \item $\{w_k\}$ contains a subsequence converging to a critical point $u_*\notin L$. In addition, if $u_*$ is isolated, then $w_k\to u_*$ as $k\to\infty$.
\end{enumerate}
\end{theorem}

\begin{proof}
The decreasing condition \eqref{eq:wp-k-1} (or \eqref{eq:swp-k-1}), condition (ii) and assumption (A1) say that, for $k=0,1,\ldots$,
\begin{equation}\label{eq:cvg-JJsdc}
  E(w_{k+1}) - E(w_k) \leq\sigma_1\alpha_kt_k\langle E'(w_k),d_k\rangle\leq -\sigma_1\delta c_1 \alpha_k\|E'(w_k)\|_{H^*}^2.
\end{equation}
Therefore, the sequence $\{E(w_k)\}$ is monotonically non-increasing. In addition, since the condition (iii) guarantees that $E(w_k)$  is bounded from below for all $k=0,1,\ldots$, the sequence $\{E(w_k)\}$ converges to some $E_{\infty}:=\inf_{k\geq0}E(w_k)$.

Then, adding up \eqref{eq:cvg-JJsdc}, we can arrive at
\begin{equation}\label{eq:cvg-sumJJ}
  \sum\limits_{k=0}^\infty \big(E(w_{k+1}) - E(w_k)\big) \leq -\sigma_1\delta c_1 \sum\limits_{k=0}^\infty \alpha_k\|E'(w_k)\|_{H^*}^2.
\end{equation}
Hence, the left-hand side of \eqref{eq:cvg-sumJJ} converges to $E_{\infty}-E(w_0)$ which is finite. This immediately leads to the conclusion (a).

To prove the conclusion (b), let $\bar{u}\in H$ be an accumulation point of the sequence $\{w_k\}$. Then, there exists a subsequence $\{w_{k_i}\}$ converging to $\bar{u}$ as $i\to\infty$. Recalling the weak version of the homeomorphism property in Lemma~\ref{lem:home}, it leads to that the corresponding subsequence $\{v_{k_i}\}$ converges to some $\bar{v}\in\mathcal{V}_0$ satisfying $\bar{u}=p(\bar{v})=t_{\bar{v}}\bar{v}+w_{\bar{v}}^L$. Lemma~\ref{lem:mono} and the condition (ii) yield that, for some $\delta>0$,
\[
  \mathrm{dist}(\bar{u},L)
  = \lim_{i\to\infty}\mathrm{dist}(w_{k_i},L)
  = \lim_{i\to\infty}t_{k_i}\|v_{k_i}^\bot\|
  \geq\delta\|v_0^\bot\|>0,
\]
and therefore $\bar{u}\notin L$.

The following is to verify that $\bar{u}$ is a critical point by taking full advantages of the curvature condition \eqref{eq:wp-k-2} (or \eqref{eq:swp-k-2}), 
which states that
\begin{equation}\label{eq:recall-wpk2}
  \frac{t_{k_i+1}}{\sqrt{1+\alpha_{k_i}^2\|d_{k_i}\|^2}} \langle E'(w_{k_i+1}),d_{k_i}\rangle
  \geq \sigma_2t_{k_i}\langle E'(w_{k_i}),d_{k_i}\rangle, \quad i=0,1,\ldots.
\end{equation}
By the contradiction argument, suppose that $\bar{u}$ is not a critical point, i.e., $E'(\bar{u})\neq0$. Since the functional $E\in C^1(H,\mathbb{R})$, one can obtain
\begin{equation}\label{eq:dEw->dEu}
  E'(w_{k_i})\to E'(\bar{u}),\quad \mbox{as }i\to\infty.
\end{equation}
As a result, for all $i$ large enough, $\|E'(w_{k_i})\|_{H^*}>\|E'(\bar{u})\|_{H^*}/2>0$. In view of the conclusion (a), it leads to
\begin{equation}\label{eq:aki->0}
  \alpha_{k_i}\to0,\quad \mbox{as }i\to\infty.
\end{equation}
In addition, the assumption (A3) states that there exists $\bar{d}\in[L,\bar{v}]^\bot$ s.t.
\begin{equation}\label{eq:dki->dbar}
  d_{k_i}\to\bar{d},\quad \mbox{as }i\to\infty.
\end{equation}
Therefore, \eqref{eq:aki->0} and \eqref{eq:dki->dbar} immediately indicate that
\begin{equation*}
  v_{k_i+1}=\frac{v_{k_i}+\alpha_{k_i}d_{k_i}}{\sqrt{1+\alpha_{k_i}^2\|d_{k_i}\|^2}}\to\bar{v}\quad\mbox{in }H,\quad \mbox{as }i\to\infty,
\end{equation*}
and 
\begin{equation}\label{eq:dEw1->dEu}
  E'(w_{k_i+1})=E'(p(v_{k_i+1}))\to E'(p(\bar{v}))= E'(\bar{u}),\quad \mbox{as }i\to\infty,
\end{equation}
holds by the continuity of $E'$ and $p$. 
Moreover, reviewing Lemma~\ref{lem:pv-tv} and the condition (ii), for some $\delta>0$, we have
\begin{equation}\label{eq:tki->tvbar}
  t_{k_i}\to t_{\bar{v}}\geq\delta>0 \;\mbox{ and }\; t_{k_i+1}\to t_{\bar{v}}\geq\delta>0,\quad \mbox{as }i\to\infty.
\end{equation}
Above all, combining \eqref{eq:dEw->dEu}-\eqref{eq:tki->tvbar} and taking $i\to\infty$ in \eqref{eq:recall-wpk2} imply
\begin{equation}\label{eq:tdEdsigma}
  t_{\bar{v}}\langle E'(\bar{u}),\bar{d}\,\rangle \geq \sigma_2t_{\bar{v}}\langle E'(\bar{u}),\bar{d}\,\rangle. 
\end{equation}
Since $E'(\bar{u})=E'(p(\bar{v}))\neq0$, it follows from the assumption (A1), \eqref{eq:dEw->dEu} and \eqref{eq:dki->dbar} that $\langle E'(\bar{u}),\bar{d}\,\rangle \leq -c_1\|E'(\bar{u})\|_{H^*}^2 < 0$ for some $c_1>0$. Thus, we have $t_{\bar{v}}\langle E'(\bar{u}),\bar{d}\,\rangle < 0$ and \eqref{eq:tdEdsigma} contradicts the fact $0<\sigma_2<1$. Consequently, the accumulation point $\bar{u}=p(\bar{v})$ is a critical point. The conclusion (b) is obtained.

Next, we prove the conclusion (c) by the contradiction argument. Suppose that
\[ \delta_1:=\liminf\limits_{k\to\infty}\|E'(w_k)\|_{H^*}>0. \]
Then, for $k$ large enough, $\|E'(w_k)\|_{H^*}\geq\delta_1/2>0$. Thus, the conclusion (a) admits 
\begin{equation}\label{eq:akgkcvg}
  \sum\limits_{k=0}^{\infty}\alpha_k\|E'(w_k)\|_{H^*}<\infty.
\end{equation}
In addition, Lemma~\ref{lem:vs} and the assumption (A2) yield that, for some $c_2>0$,
\begin{equation}\label{eq:vkakdk}
  \|v_{k+1}-v_k\|=\|v_k(\alpha_k)-v_k\|\leq \alpha_k\|d_k\|\leq c_2\alpha_k\|E'(w_k)\|_{H^*},\quad k=0,1,\ldots.
\end{equation}
Combining \eqref{eq:akgkcvg} and \eqref{eq:vkakdk} results in $\sum\limits_{k=0}^{\infty}\|v_{k+1}-v_k\|<\infty$. Therefore, $\{v_k\}$ is a Cauchy sequence in the closed subset $\mathcal{V}_0$. Immediately, the completeness of the closed subset $\mathcal{V}_0$ implies that there exists $\bar{v}\in\mathcal{V}_0$ s.t. $v_k\to\bar{v}$ as $k\to\infty$. Further, by the continuity of $p$ and $E'$, we have $w_k=p(v_k)\to p(\bar{v})$ with $p(\bar{v})$ the accumulation point, and $E'(w_k)\to E'(p(\bar{v}))$ as $k\to\infty$. Hence, there holds
\[ \|E'(p(\bar{v}))\|_{H^*}=\lim_{k\to\infty}\|E'(w_k)\|_{H^*}=\delta_1>0. \]
This is a contradiction to the conclusion (b). Thus, the conclusion (c) is proved.

The rest is to prove the conclusion (d). Since $\{E(w_k)\}$ converges to $E_{\infty}$ by the proof of the conclusion (a), according to the conclusion (c), there exists a subsequence $\{w_{k_i}\}$ s.t. $E(w_{k_i})\to E_{\infty}$ and $E'(w_{k_i})\to0$ as $i\to\infty$. By the (PS) condition, $\{w_{k_i}\}$ possesses a subsequence, still denoted by $\{w_{k_i}\}$, that converges to a critical point $u_*\in H$. In addition, according to the conclusion (b), $u_*\notin L$ holds. Finally, under the assumption that $u_*$ is isolated and following the analogous lines in the proof of Theorem 2.4 in \cite{Z2017CAMC} for the global convergence, the proof is completed.
\end{proof}

\section{Descent directions}\label{sec:dk}

In this section, we propose two specific types of descent directions for implementing Algorithm~\ref{alg:wplmm} in details. One is the PSD direction and the other is the CG-type descent direction. We use the same notations as those in Algorithm~\ref{alg:wplmm} for subsequent discussions in this section, unless specified.

\subsection{Preconditioned steepest descent direction}

The gradient of $E$ at $w_k=p(v_k)$, denoted by $g_k=\nabla E(w_k)$, is defined by
\begin{equation}\label{eq:grad-def}
  (g_k,\phi) = \langle E'(p(v_k)),\phi\rangle, \quad \forall \phi\in H.
\end{equation}
From the Riesz representation theorem, the gradient $g_k\in H$ exists uniquely and $\|g_k\|=\|E'(p(v_k))\|_{H^*}$. Consider the following PSD direction
\begin{equation}\label{eq:psdd}
  d_k=-T_kg_k,\quad k=0,1,\ldots,
\end{equation}
where $T_k=T(v_k)$, called a preconditioner at $v_k$, is a positive-definite and self-adjoint bounded linear operator on $H$ with an invariant subspace $[L, v_k]^\bot$. Assume that $T_k=T(v_k)$ satisfies 
\begin{enumerate}[\rm{(T}1{)}]
\item for some $c_3>0$, $\|T(v_k)u\|\leq c_3\|u\|$, $\forall u\in[L, v_k]^\bot$; 
\item for some $c_4>0$, $(T(v_k)u,u)\geq c_4\|u\|^2$, $\forall u\in[L, v_k]^\bot$; 
\item $T(v)$ is continuous at the accumulation point $\bar{v}$ of $\{v_k\}$ s.t. $E'(p(\bar{v}))\neq0$.
\end{enumerate}

The following theorem provides the global convergence result of Algorithm~\ref{alg:wplmm} with $d_k$ taken as the PSD direction \eqref{eq:psdd}.

\begin{theorem}\label{thm:cvg-psd}
Suppose $E\in C^1(H,\mathbb{R})$ and let $p$ be a peak selection of $E$ w.r.t. $L$, $\{v_k\}$ and $\{w_k\}$ be sequences generated by Algorithm~\ref{alg:wplmm} with $d_k=-T_kg_k$ the PSD direction defined in \eqref{eq:psdd} and $T_k=T(v_k)$ satisfying assumptions {\rm(T1)}-{\rm(T3)}. If 
{\rm(i)} $p$ is locally Lipschitz continuous on $\mathcal{V}_0$;
{\rm(ii)} $t_k\geq\delta$ for some $\delta>0$, $\forall\,k=0,1,\ldots$; and
{\rm(iii)} $\inf_{v\in \mathcal{V}_0}E(p(v))>-\infty$,
then those conclusions {\rm(a)}-{\rm(d)} in Theorem~\ref{thm:cvg-wplmm} hold.
\end{theorem}
\begin{proof}
Since $g_k\in[L, v_k]^\bot$ from Lemma~\ref{lem:orth} and $[L, v_k]^\bot$ is an invariant subspace of $T_k=T(v_k)$, we have $d_k=-T_kg_k\in[L, v_k]^\bot$. In order to prove the conclusion, it suffices to verify assumptions (A1)-(A3) of Theorem~\ref{thm:cvg-wplmm}. In fact, firstly, the assumption {\rm(T2)} states that
\begin{equation}
  \langle E'(w_k),d_k\rangle= -(g_k,T_kg_k)\leq -c_4\|g_k\|^2 =-c_4\|E'(w_k)\|_{H^*}^2,\quad k=0,1,\ldots,
\end{equation}
which is (A1) (with $c_4=c_1$). Further, by the assumption {\rm(T1)},
\begin{equation}
  \|d_k\|=\|T_kg_k\|\leq c_3\|g_k\|= c_3\|E'(w_k)\|_{H^*},\quad k=0,1,\ldots.
\end{equation}
Thus, (A2) is verified by taking $c_3=c_2$. Finally, (A3) directly follows from the assumption {\rm(T3)} and the continuity of $E'$ and $p$.
\end{proof}

Taking $T_k$ in Theorem~\ref{thm:cvg-psd} simply as the identity operator on $H$ yields the following corollary, which draws the global convergence of Algorithm~\ref{alg:wplmm} with the standard steepest descent direction utilized at each iterative step.

\begin{corollary}\label{thm:cvg-sd}
Suppose $E\in C^1(H,\mathbb{R})$ and let $p$ be a peak selection of $E$ w.r.t. $L$, $\{v_k\}$ and $\{w_k\}$ be sequences generated by Algorithm~\ref{alg:wplmm} with $d_k=-g_k$. If
{\rm(i)} $p$ is locally Lipschitz continuous on $\mathcal{V}_0$;
{\rm(ii)} $t_k\geq\delta$ for some $\delta>0$, $\forall\,k=0,1,\ldots$; and
{\rm(iii)} $\inf_{v\in \mathcal{V}_0}E(p(v))>-\infty$,
then those conclusions {\rm(a)}-{\rm(d)} in Theorem~\ref{thm:cvg-wplmm} hold.
\end{corollary}

\begin{remark}
It is noted that Theorem 2.4 in \cite{Z2017CAMC}, Theorem 5.1 in \cite{LXY2021CMS} and Corollary~\ref{thm:cvg-sd} provided respectively the global convergence of the NA-LMM, NG-LMM and NWP-LMM with the steepest descent direction. Consequently, mathematical justifications of LMMs combined with several typical inexact normalized step-size search rules for the steepest descent direction have been systematically established. 
\end{remark}

\subsection{Conjugate gradient-type direction}
For $k=0,1,\ldots$, denote $v_k=v_k^L+v_k^\bot$ with $v_k^L\in L$ and $0\neq v_k^\bot\in L^\bot$. Similar to the construction of the nonlinear CG method in the optimization theory (see, e.g., \cite{DaiYuan2006,HZ2006PJO}), we consider the following CG-type direction for Algorithm~\ref{alg:wplmm}:
\begin{equation}\label{eq:cg-dk}
 d_k= \begin{cases}
          -g_0, & k=0, \\
          -g_k + \beta_k \Pi_k d_{k-1}, & k\geq1.
        \end{cases}
\end{equation}
Here, $\beta_k\in\mathbb{R}$ is a parameter to be determined, $g_k=\nabla E(w_k)$ is the gradient of $E$ at $w_k$ defined in \eqref{eq:grad-def}, and $\Pi_k$ is the orthogonal projection from $H$ onto $[L, v_k]^\bot$. 
Since $g_k\in[L, v_k]^\bot$ from Lemma~\ref{lem:orth}, we have $d_k\in[L, v_k]^\bot$ for all $k\geq 0$. In addition, according to the definition of $\Pi_k$, it holds that $(g_k,\Pi_kd_{k-1})=(g_k, d_{k-1})$, $k\geq 1$. Thus, the CG-type direction \eqref{eq:cg-dk} satisfies
\begin{equation}\label{eq:cg-gkdk}
  (g_k, d_k) = -\|g_k\|^2+\beta_k(g_k, d_{k-1}),\quad k\geq 1.
\end{equation}

\begin{remark}
We remark here that $\Pi_kd_{k-1}$ can be explicitly expressed as
\begin{align}\label{eq:Pid}
\Pi_kd_{k-1} = d_{k-1}-\|v_k^\bot\|^{-2}\big(d_{k-1},v_k^\bot\big)v_k^\bot,\quad k\geq1.
\end{align}
Actually, by applying the facts that $d_{k-1}\in[L, v_{k-1}]^\bot\subset L^\bot$, $\Pi_kd_{k-1}\in[L, v_k]^\bot\subset L^\bot$, and $\mathrm{Id}-\Pi_k$ is an orthogonal projection onto $([L, v_k]^\bot)^\bot=L\oplus[v_k^\bot]$, where $\mathrm{Id}$ is the identity operator on $H$ and $\oplus$ denotes the direct sum, we can conclude that
\[ d_{k-1}-\Pi_kd_{k-1}\in \big((L\oplus[v_k^\bot])\cap L^\bot\big)=[v_k^\bot]. \] 
Thus, $d_{k-1}-\Pi_kd_{k-1}=c\,v_k^\bot$ for some $c\in\mathbb{R}$. Taking the inner product with $v_k^\bot$ and noting that $\big(\Pi_kd_{k-1},v_k^\bot\big)=\big(\Pi_kd_{k-1},v_k-v_k^L\big)=0$, we can obtain
\[ \big(d_{k-1},v_k^\bot\big)=\big(d_{k-1}-\Pi_kd_{k-1},v_k^\bot\big)=c\|v_k^\bot\|^2, \]
yielding $c=\|v_k^\bot\|^{-2}\big(d_{k-1},v_k^\bot\big)$. Consequently, the expression \eqref{eq:Pid} is true. 
\end{remark}

The following lemma shows that, if the exact step-size search rule is applied in the previous iteration, the CG-type direction \eqref{eq:cg-dk} with an arbitrary parameter $\beta_k$ is a descent direction.
\begin{lemma}
For $k\geq 1$, let $v_k=v_{k-1}(\alpha_{k-1})$ with $d_k$ defined in \eqref{eq:cg-dk} and $\alpha_{k-1}>0$ a local minimizer of $E(p(v_{k-1}(\alpha)))$ along $\{\alpha:\alpha>0\}$. If $E\in C^1(H,\mathbb{R})$, $p$ is locally Lipschitz continuous around $v_k$, and $t_k>0$, then $(g_k, d_k)=-\|g_k\|^2$.
\end{lemma}
\begin{proof}
From Lemma~\ref{lem:dEpvs}, $E(p(v_{k-1}(\alpha)))$ is continuously differentiable at $\alpha=\alpha_{k-1}$ and
\[
 \frac{\mathrm{d}}{\mathrm{d} \alpha}E(p(v_{k-1}(\alpha)))\Big|_{\alpha=\alpha_{k-1}}
 = \frac{t_k}{\sqrt{1+\alpha_{k-1}^2\|d_{k-1}\|^2}}\langle E'(p(v_k)),d_{k-1} \rangle.
\]
By applying the facts that $t_k>0$ and $E(p(v_{k-1}(\alpha)))$ attains its local minimum at $\alpha_{k-1}$, we have $\langle E'(p(v_k)),d_{k-1}\rangle=(g_k, d_{k-1})=0$. Thus, the conclusion follows from \eqref{eq:cg-gkdk} immediately.
\end{proof}

Due to the fact that the exact step-size search rule is quite expensive in practical computations, we focus on how to ensure that the CG-type direction $d_k$ defined in \eqref{eq:cg-dk} is a descent direction when a suitable inexact step-size search rule is used.

Inspired by the well-known Fletcher-Reeves CG method \cite{FR1964CJ} in the optimization theory, we set
\begin{equation}\label{eq:beta-fr}
  \beta_k^{\text{\rm FR-like}} = \frac{\gamma_k\|g_k\|^2}{\|g_{k-1}\|^2},\quad k\geq 1,
\end{equation}
where $\gamma_k=\hat{t}_k/t_{k-1}$ with $\hat{t}_k=t_k/\sqrt{1+\alpha_{k-1}^2\|d_{k-1}\|^2}$. It can be verified that the CG-type direction \eqref{eq:cg-dk} with $\beta_k=\beta_k^{\text{\rm FR-like}}$ is a descent direction if the step-size $\alpha_k$ in each iteration satisfies the normalized strong Wolfe-Powell-type step-size search rule \eqref{eq:swp-k} with $\sigma_2\in(0,1/2)$, as stated in the following lemma. The proof follows the lines of the proof for Theorem~1 in \cite{Al-Baali1985}.

\begin{lemma}\label{lem:cgfr-desc}
For $k=0,1,\ldots$, if $g_k\neq0$, $d_k$ in Algorithm~\ref{alg:wplmm} is defined in \eqref{eq:cg-dk} with $\beta_k=\beta_k^{\text{\rm FR-like}}$, and $\alpha_k$ is determined by the normalized strong Wolfe-Powell-type step-size search rule \eqref{eq:swp-k} with $\sigma_2\in(0,1/2)$, then
\begin{equation}\label{eq:cgfr-ineq}
  -\frac{1-\sigma_2^{k+1}}{1-\sigma_2}\leq\frac{(g_k, d_k)}{\|g_k\|^2} \leq -\frac{1-2\sigma_2+\sigma_2^{k+1}}{1-\sigma_2},\quad k=0,1,\ldots.
\end{equation}
Consequently,
\begin{equation}\label{eq:cgfr-desc}
(g_k, d_k)<0,\quad k=0,1,\ldots.
\end{equation}
\end{lemma}
\begin{proof}
For $k=0$, $d_0=-g_0$, conclusions \eqref{eq:cgfr-ineq} and \eqref{eq:cgfr-desc} are obvious. According to the inductive argument, suppose that conclusions \eqref{eq:cgfr-ineq} and \eqref{eq:cgfr-desc} hold for $k-1$ ($k\geq1$). In view of the definition of $\beta_k^{\text{\rm FR-like}}$ in \eqref{eq:beta-fr}, the fact \eqref{eq:cg-gkdk} yields
\[
  \frac{(g_k, d_k)}{\|g_k\|^2} = -1 + \gamma_k\frac{(g_k, d_{k-1})}{\|g_{k-1}\|^2},\quad k\geq 1.
\]
Using the inductive assumption \eqref{eq:cgfr-desc} for $k-1$  ($k\geq1$), the normalized strong Wolfe-Powell-type step-size search rule \eqref{eq:swp-k} states that
\[
  \gamma_k\left|(g_k, d_{k-1})\right| \leq -\sigma_2 (g_{k-1},d_{k-1}),
\]
and therefore,
\[
  -1 + \sigma_2\frac{(g_{k-1},d_{k-1})}{\|g_{k-1}\|^2}
  \leq \frac{(g_k, d_k)}{\|g_k\|^2}
  \leq -1 - \sigma_2\frac{(g_{k-1},d_{k-1})}{\|g_{k-1}\|^2}.
\]
Further, by the inductive assumption \eqref{eq:cgfr-ineq} for $k-1$ ($k\geq1$), we can arrive at
\[
  -\frac{1-\sigma_2^{k+1}}{1-\sigma_2}
  = -1 - \sigma_2\frac{1-\sigma_2^k}{1-\sigma_2}
  \leq \frac{(g_k, d_k)}{\|g_k\|^2}
  \leq -1 + \sigma_2\frac{1-\sigma_2^k}{1-\sigma_2}
  = -\frac{1-2\sigma_2+\sigma_2^{k+1}}{1-\sigma_2},
\]
which leads to the conclusion \eqref{eq:cgfr-ineq} for $k\geq 1$. Since $\sigma_2\in(0,1/2)$, it immediately follows that $(g_k, d_k)<0$, $k\geq 1$. The proof is finished by the inductive argument.
\end{proof}

\begin{remark}
Under assumptions in Lemma~\ref{lem:cgfr-desc}, we have from \eqref{eq:cgfr-ineq} that $(g_k, d_k)\leq-c_1\|g_k\|^2$ with $c_1=(1-2\sigma_2)/(1-\sigma_2)>0$, i.e., the CG-type direction \eqref{eq:cg-dk} with $\beta_k=\beta_k^{\text{\rm FR-like}}$ satisfies the assumption {\rm(A1)} in Theorem~\ref{thm:cvg-wplmm}. Moreover, under the same assumptions, these conclusions in Lemma~\ref{lem:cgfr-desc} can be extended directly to any choice of $\beta_k$ satisfying $|\beta_k| \leq \beta_k^{\text{\rm FR-like}}$. 
\end{remark}

\begin{remark}
It is currently unclear whether assumptions {\rm(A2)} and {\rm(A3)} in Theorem~\ref{thm:cvg-wplmm} hold for the NWP-LMM with the CG-type descent direction \eqref{eq:cg-dk}. Thus, the global convergence of it has not been verified yet and will be our future work. Indeed, the NWP-LMM with the CG-type descent direction \eqref{eq:cg-dk} is very efficient for finding multiple solutions of semilinear elliptic PDEs, compared to traditional LMMs and the NWP-LMM with the steepest descent direction, which will be shown in the section~\ref{sec:numer}. Actually, several different constructions of CG-type descent directions can also be designed based on similar ideas of various CG methods in the optimization theory, which can be found, e.g., in \cite{DaiYuan2006,HZ2006PJO}.
\end{remark}

\section{Numerical examples}\label{sec:numer}

In this section, we apply our NWP-LMM to find multiple unstable solutions of the semilinear elliptic boundary value problem (BVP) 
\begin{equation}\label{eq:P1}
\left\{
\begin{aligned}
  -\Delta u(\mathbf{x})+a(\mathbf{x})u(\mathbf{x}) &=f(\mathbf{x},u(\mathbf{x})), & \mathbf{x}\in\Omega, \\
  u(\mathbf{x}) &=0, & \mathbf{x}\in\partial\Omega,
\end{aligned}
\right.
\end{equation}
where $\Omega$ is a bounded domain in $\mathbb{R}^N$ with a Lipschitz boundary $\partial\Omega$, $a\in L^{\infty}(\Omega)$ and $a(\mathbf{x})\geq0$ ($\mathbf{x}\in\Omega$), and $f:\bar{\Omega}\times\mathbb{R}\to\mathbb{R}$ satisfies the following standard hypotheses \cite{Rabinowitz1986}:
\begin{enumerate}[(f1)]
%[(f1)-(f4)]
  \item $f(\mathbf{x},\xi)$ is locally Lipschitz on $\bar{\Omega}\times\mathbb{R}$;
  \item there is a constant $c>0$ s.t. $|f(\mathbf{x},\xi)|\leq c(1+|\xi|^{s-1})$ for some $s\in(2,2^*)$, where $2^*:=2N/(N-2)$, if $N\geq3$; and $2^*:=\infty$, if $N=1,2$;
  \item there are constants $\mu>2$ and $R>0$ s.t., for $|\xi|\geq R$, $0<\mu F(\mathbf{x},\xi)\leq f(\mathbf{x},\xi)\xi$, where $F(\mathbf{x},\xi)=\int_0^{\xi}f(\mathbf{x},t)\mathrm{d} t$;
  \item $f(\mathbf{x},\xi)=o(\xi)$ as $\xi\to0$.
\end{enumerate}

Define $H=H_0^1(\Omega)$ with the $a$-dependent inner product and norm as
\begin{equation}\label{eq:a-nrmipd}
  (u, v)_a=\int_{\Omega} \Big(\nabla u(\mathbf{x})\cdot\nabla v(\mathbf{x})+a(\mathbf{x})u(\mathbf{x})v(\mathbf{x})\Big) \mathrm{d}\mathbf{x},\quad \|u\|_a=\sqrt{(u,u)_a}.
\end{equation}
Since $a(\mathbf{x})$ is nonnegative and uniformly bounded, the norm $\|\cdot\|_a$ is equivalent to the usual norm in $H_0^1(\Omega)$, i.e., $\|u\|=\big(\int_{\Omega}|\nabla u(\mathbf{x})|^2\mathrm{d} \mathbf{x}\big)^{1/2}$. The variational energy functional associated to the BVP \eqref{eq:P1} is given as
\begin{equation*}
  E(u) = \!\int_{\Omega}\left(\frac12\big(|\nabla u(\mathbf{x})|^2+a(\mathbf{x})|u(\mathbf{x})|^2\big) - F(\mathbf{x},u(\mathbf{x})) \right)\! \mathrm{d}\mathbf{x}
  = \frac12\|u\|_a^2-\int_{\Omega}F(\mathbf{x},u(\mathbf{x}))\mathrm{d}\mathbf{x}.
\end{equation*}

It is well known that, under hypotheses (f1)-(f4), $E\in C^1(H,\mathbb{R})$ and satisfies the (PS) condition. Each critical point of $E$ is a weak solution and also a classical solution to the BVP \eqref{eq:P1} \cite{Rabinowitz1986}. In addition, $u\equiv0$ is a local minimizer of $E$. Moreover, in any finite-dimensional subspace of $H$, $E(u)\to-\infty$ uniformly as $\|u\|_a\to\infty$. Hence, for any finite-dimensional closed subspace $L$, the peak mapping $P$ of $E$ w.r.t. $L$ is nonempty. According to \cite{LZ2001SISC}, we assume, in addition to hypotheses (f1)-(f4), that
\begin{enumerate}[(f1)]
\setcounter{enumi}{4} %[(f5)]
\item $f(x,\xi)/|\xi|$ is increasing w.r.t. $\xi$ on $\mathbb{R}\backslash\{0\}$.
\end{enumerate}
For $L=\{0\}$, as shown in \cite{LZ2001SISC}, under hypotheses (f1)-(f5), $E$ has only one local maximizer in any direction, i.e., $E$ has a unique peak selection $p(v)=t_vv$ ($\forall v\in S_H$) w.r.t. $L=\{0\}$. Moreover, there exists $\delta>0$ s.t. $t_v=\|p(v)\|\geq\delta$ for any $v\in S_H$, that is exactly the separation condition in Theorem~\ref{thm:cvg-wplmm} when $L=\{0\}$. For any finite-dimensional closed subspace $L$, the uniqueness of the peak selection of $E$ w.r.t. $L$ implies its continuity. As a result, the unique peak selection $p$ w.r.t. $L=\{0\}$ is continuous on $S_H$. 
Moreover, if conditions (f1)-(f5) hold and
\begin{enumerate}[(f1)]
\setcounter{enumi}{5} %[(f6)]
\item $f(x,\xi)$ is $C^1$ and there exists a constant $\tilde{c}>0$ s.t. $|f_{\xi}(x,\xi)|\leq \tilde{c}(1+|\xi|^{s-2})$, for $s$ as specified in (f2), 
\end{enumerate}
then the unique peak selection $p$ w.r.t. $L=\{0\}$ is $C^1$ \cite{LZ2001SISC}.

Due to the limit of the length of the paper, numerical experiments mainly focus on the following three cases of the BVP \eqref{eq:P1} on a 2D domain (a square or dumbbell-shaped domain). 
\begin{enumerate}[{Case}~1.]
\item ({\em NLSE in the focusing regime} \cite{DGPS1999RMP}) 
  $f(\mathbf{x}, u)=u^3$, $a(\mathbf{x})=\omega|\mathbf{x}|^2$ with $\omega>0$ and $|\mathbf{x}|:=\sqrt{x_1^2+x_2^2}$ for all $\mathbf{x}=(x_1,x_2)\in\Omega=(-1,1)^2\subset\mathbb{R}^2$.
\item ({\em H\'{e}non equation} \cite{CZN2000IJBC,LZ2002SISC}) 
  $a(\mathbf{x})=0$, $f(\mathbf{x}, u)=|\mathbf{x}|^{\ell}u^3$ ($\ell\geq0$) and $\Omega=(-1,1)^2$.
\item ({\em Chandrasekhar equation} \cite{CZN2000IJBC}) 
  $a(\mathbf{x})=0$, $f(\mathbf{x}, u)=(u^2+2u)^{3/2}$, $u\geq0$ and $\Omega$ is a 2D dumbbell-shaped domain as described later.  
\end{enumerate}
It is clear that hypotheses (f1)-(f6) hold for all functions $f$ in Cases~1-3. In addition, our approach is efficient for different domains such as a L-shaped domain, a ball or other complex domains in high dimensions.

In our numerical experiments, the initial ascent direction $v_0$ is taken as $v_0=\tilde{v}_0/\|\tilde{v}_0\|_a$ with $\tilde{v}_0$ the solution to the Poisson problem
\begin{equation}\label{eq:poisson-v0}
\left\{
\begin{aligned}
  -\Delta\tilde{v}_0(\mathbf{x}) &= \mathbf{1}_{\Omega_1}(\mathbf{x})-\mathbf{1}_{\Omega_2}(\mathbf{x}), & \mathbf{x}\in\Omega, \\
  \tilde{v}_0(\mathbf{x}) &= 0, & \mathbf{x}\in\partial\Omega,
\end{aligned}
\right.
\end{equation}
where $\mathbf{1}_.(\mathbf{x})$ is the indicator function and $\Omega_1$, $\Omega_2$ are two selected disjoint subdomains of $\Omega$ to control the convexity of $v_0$. A numerical solution is reached at $w_k=p(v_k)$ by the NWP-LMM algorithm when $\|g_k\|_a=\|\nabla E(w_k)\|_a\leq 10^{-5}$ and $\max_{\mathbf{x}\in\Omega}|\Delta w_k(\mathbf{x})-a(\mathbf{x})w_k(\mathbf{x})+f(\mathbf{x},w_k(\mathbf{x}))|\leq 5\times 10^{-5}$. 
Particularly, it is necessary to explain more about how to numerically compute the gradient of the energy functional and a peak selection in numerical experiments. According to the definition of the gradient $g_k=\nabla E(w_k)$ of $E$ at $w_k\in H$ in \eqref{eq:grad-def} and the definition of the inner product $(\cdot,\cdot)_a$ in \eqref{eq:a-nrmipd}, 
the gradient $g_k$ can be expressed as $g_k=w_k-\phi_k$ with $\phi_k$ determined by the linear elliptic BVP
\begin{equation}
\left\{
\begin{aligned}
  -\Delta \phi_k(\mathbf{x})+a(\mathbf{x})\phi_k(\mathbf{x}) &= f(\mathbf{x},w_k(\mathbf{x})), & \mathbf{x}\in\Omega,\\
  \phi_k(\mathbf{x}) &=0, & \mathbf{x}\in\partial\Omega,
\end{aligned}
\right.
\end{equation}
which can be solved efficiently by a standard numerical method, such as the finite element method (FEM) or the finite difference method. In our numerical code, {\ttfamily assempde}, a FEM-based subroutine provided by the MATLAB PDE Toolbox, is implemented to accomplish this task. The square domain in Cases~1-2 and the dumbbell-shaped domain in Case~3 are, respectively, discretized with 32768 and 15552 triangular elements. In addition, the computation of a peak selection $p(v_k)$ of $E$ at $v_k\in S_H$ w.r.t. a given $(n-1)$-dimensional closed subspace $L$ ($n\geq1$) is an optimization problem in the $n$-dimensional half subspace $[L,v_k]$. To do this, a MATLAB subroutine {\tt fminunc} with the termination tolerance on the first-order optimality ${\tt Tol}=10^{-8}$ is called in our numerical code.

Further, for the convenience of numerical comparisons, we introduce the following notations for three different algorithms:
\begin{itemize}
  \item[$\bullet$] {\tt SD-Armijo}: the traditional LMM algorithm by utilizing the steepest descent direction $d_k=-g_k$ and the normalized Armijo-type step-size search rule \cite{LZ2002SISC,XYZ2012SISC}. At each iterative step of this algorithm, the step-size $\alpha_k$ is chosen as
\begin{equation}\label{eq:ak-armijo}
\alpha_k = \max_{m\in\mathbb{N}}\left\{\lambda\rho^m\,:\, E(p(v_k(\lambda\rho^m))) \leq E(p(v_k)) - \sigma\lambda\rho^mt_k\|g_k\|^2\right\},
\end{equation}
where constants $\sigma$, $\lambda$ are fixed as $\sigma=10^{-4}$ and $\lambda=0.1$.
  \item[$\bullet$] {\tt SD-StrongWolfe}: the NWP-LMM algorithm by utilizing the steepest descent direction $d_k=-g_k$ and the normalized strong Wolfe-Powell-type step-size search rule \eqref{eq:swp-k} with constants $\sigma_1=10^{-4}$ and $\sigma_2=0.4$.
  \item[$\bullet$] {\tt CG-StrongWolfe}: the NWP-LMM algorithm by utilizing the CG-type direction \eqref{eq:cg-dk} with $\beta_k=\beta_k^{\text{\rm FR-like}}$ given in \eqref{eq:beta-fr} and the normalized strong Wolfe-Powell-type step-size search rule \eqref{eq:swp-k} with constants $\sigma_1=10^{-4}$ and $\sigma_2=0.4$.
\end{itemize}
It is worthwhile to point out that it is generally not feasible to directly combine the CG-type direction and the Armijo-type step-size search rule in the LMM (i.e., $d_k$ is not guaranteed to be a descent direction) as observed numerically.

\subsection{Numerical results for the nonlinear Schr\"{o}dinger equation}\label{sec:numer-NLSE}

In this subsection, we report numerical results of Case~1, i.e., the NLSE in the focusing regime on the square domain $\Omega=(-1,1)^2$ as
\begin{equation}
\left\{
\begin{aligned}
  -\Delta u(\mathbf{x})+\omega|\mathbf{x}|^2u(\mathbf{x}) &=u^3(\mathbf{x}), & \mathbf{x}\in\Omega, \\
  u(\mathbf{x}) &=0, & \mathbf{x}\in\partial\Omega.
\end{aligned}
\right.
\end{equation}
Taking $\omega=8$, limited by paper length, only ten different solutions labeled by $u_1,u_2,\ldots,u_{10}$ are shown in Fig.~\ref{fig:NLSrect10sols} for their profiles and features. The corresponding information on the support space $L$, subdomains $\Omega_1$ and $\Omega_2$ used in \eqref{eq:poisson-v0}, and energy values of these solutions is listed in Table~\ref{tab:NLSrect10sols}. In addition, numerical comparisons of the {\tt SD-StrongWolfe}, {\tt CG-StrongWolfe} and {\tt SD-Armijo} in terms of CPU times for computing these solutions of the NLSE are presented in Fig.~\ref{fig:NLScpubar}.

From Figs.~\ref{fig:NLSrect10sols}-\ref{fig:NLScpubar}, Table~\ref{tab:NLSrect10sols} and additional results not shown here, we observe that three LMM algorithms considered are effective for finding multiple solutions of the NLSE in the focusing regime and the {\tt CG-StrongWolfe} shows the best performance among its LMM companions. In addition, $u_1$ (Fig.~\ref{fig:NLSrect10sols}(a)) is the only positive solution with the lowest energy value, i.e., it is the ground state solution. 

\begin{table}[!ht]
\centering
\caption{The information to corresponding solutions of the NLSE in Fig.~\ref{fig:NLSrect10sols}.}
\label{tab:NLSrect10sols}
\small
\begin{tabular}{lrllc}
\hline
 $u_n$ & ~~~$E(u_n)$~ & ~~$L$~~~~~ & ~$\Omega_1$ ($\Omega_2=\Omega\backslash\Omega_1$) & Graphics \\
\hline
 $u_1$    &  14.7889 & ~$\{0\}$             & ~$\Omega$              & Fig.~\ref{fig:NLSrect10sols}(a) \\
 $u_2$    &  73.8223 & ~$[u_1]$             & $\{x_1>0\}$            & Fig.~\ref{fig:NLSrect10sols}(b) \\
 $u_3$    &  73.8223 & ~$[u_1]$             & $\{x_2>0\}$            & Fig.~\ref{fig:NLSrect10sols}(c) \\
 $u_4$    &  70.9151 & ~$[u_1]$             & $\{x_1+x_2>0\}$        & Fig.~\ref{fig:NLSrect10sols}(d) \\
 $u_5$    &  70.9151 & ~$[u_1]$             & $\{x_1-x_2>0\}$        & Fig.~\ref{fig:NLSrect10sols}(e) \\
 $u_6$    & 210.0238 & ~$[u_1,u_2]$         & $\{|x_1|>0.2\}$        & Fig.~\ref{fig:NLSrect10sols}(f) \\
 $u_7$    & 178.2474 & ~$[u_1,u_4]$         & $\{|x_1+x_2|>0.3\}$    & Fig.~\ref{fig:NLSrect10sols}(g) \\
 $u_8$    & 213.6423 & ~$[u_1,u_2,u_3]$     & $\{x_1x_2>0\}$         & Fig.~\ref{fig:NLSrect10sols}(h) \\
 $u_9$    & 243.2646 & ~$[u_1,u_4,u_5]$     & $\{|x_1|>|x_2|\}$      & Fig.~\ref{fig:NLSrect10sols}(i) \\
 $u_{10}$ & 306.4755 & ~$[u_1,u_2,u_3,u_8]$ & $\{x_1^2+x_2^2>0.25\}$ & Fig.~\ref{fig:NLSrect10sols}(j) \\
\hline
\end{tabular}
\end{table}

\begin{figure}[!t]
\centering
\includegraphics[width=.19\textwidth]{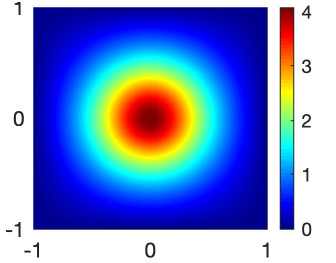}
\includegraphics[width=.19\textwidth]{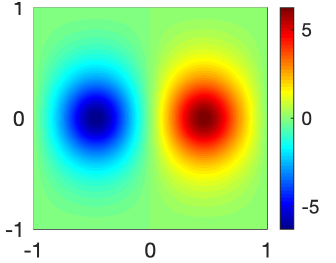}
\includegraphics[width=.19\textwidth]{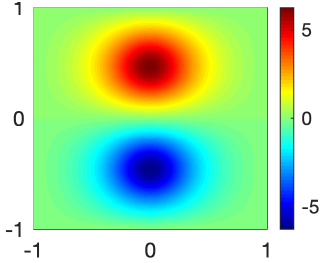}
\includegraphics[width=.19\textwidth]{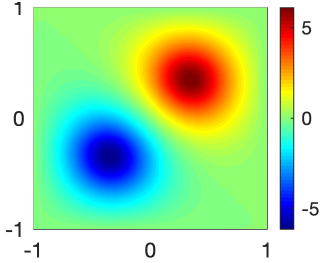}
\includegraphics[width=.19\textwidth]{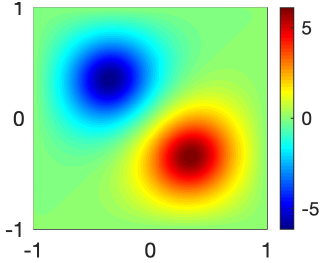} \\
\hspace*{.067\textwidth} (a)\hfill (b)\hfill (c)\hfill (d)\hfill (e) \hspace*{.085\textwidth} \\[5pt]
\includegraphics[width=.19\textwidth]{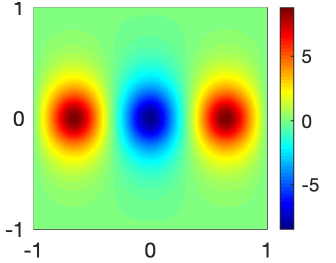}
\includegraphics[width=.19\textwidth]{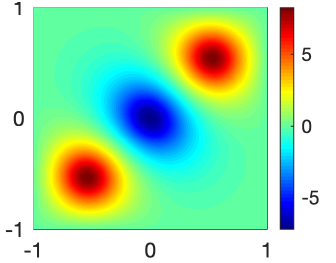}
\includegraphics[width=.19\textwidth]{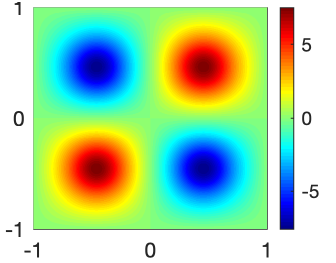}
\includegraphics[width=.19\textwidth]{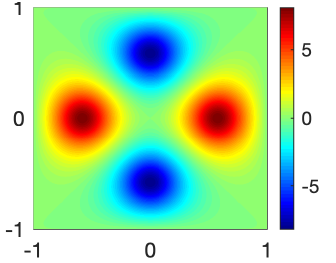}
\includegraphics[width=.19\textwidth]{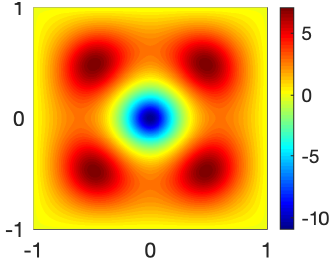} \\
\hspace*{.067\textwidth} (f)\hfill (g)\hfill (h)\hfill (i)\hfill (j) \hspace*{.085\textwidth} \\
\caption{Profiles of ten different solutions of the NLSE: 
(a) the single-peak positive (ground state) solution $u_1$ concentrated mainly on the center of the domain; 
(b)-(j) nine multi-peak sign-changing solutions $u_2\sim u_{10}$.}
\label{fig:NLSrect10sols}
\end{figure}

\begin{figure}[!t]
\centering
\includegraphics[width=.75\textwidth]{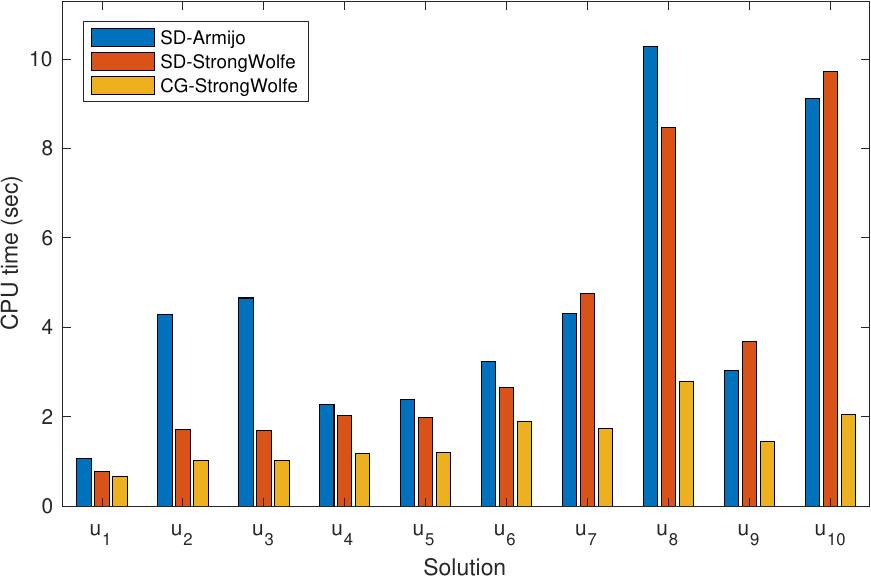}
\caption{Comparison of CPU time of LMMs for finding solutions of the NLSE in Fig.~\ref{fig:NLSrect10sols}.}
\label{fig:NLScpubar}
\end{figure}

\subsection{Numerical results for the H\'{e}non equation}\label{sec:numer-Henon}

Now, we report numerical results of Case~2, i.e., the H\'{e}non equation on the square domain $\Omega=(-1,1)^2$  as
\begin{equation}
\left\{
\begin{aligned}
  -\Delta u(\mathbf{x}) &=|\mathbf{x}|^\ell u^3(\mathbf{x}), & \mathbf{x}\in\Omega, \\
  u(\mathbf{x}) &=0, & \mathbf{x}\in\partial\Omega.
\end{aligned}
\right.
\end{equation}
 Fix $\ell=6$,  due to the space limitation, only twelve solutions we obtained, labeled by $u_1,u_2,\ldots,u_{12}$, are displayed in Fig.~\ref{fig:Henonrect12sols} for their profiles and features. The corresponding information on the support space $L$, subdomains $\Omega_1$ and $\Omega_2$ used in \eqref{eq:poisson-v0}, and energy values of these solutions is listed in Table~\ref{tab:Henonrect12sols}. In addition, numerical comparisons of the {\tt SD-StrongWolfe}, {\tt CG-StrongWolfe} and {\tt SD-Armijo} in terms of CPU times for computing these solutions of the H\'{e}non equation are provided in Fig.~\ref{fig:Henoncpubar}. 

From Figs.~\ref{fig:Henonrect12sols}-\ref{fig:Henoncpubar}, Table~\ref{tab:Henonrect12sols} and additional results not shown here, we observe that three LMM algorithms considered can effectively find multiple solutions of the H\'{e}non equation. As expected, the {\tt CG-StrongWolfe} also shows the best performance among its LMM companions. In addition, positive solutions are not unique in the case of $\ell=6$.

\begin{table}[!t]
\centering
\caption{The information to corresponding solutions of the H\'{e}non equation in Fig.~\ref{fig:Henonrect12sols}.}
\label{tab:Henonrect12sols}
\small
\begin{tabular}{lrlllc}
\hline
$u_n$ & ~~~$E(u_n)$~ & ~~$L$~~~~~ & $~\Omega_1$ & ~$\Omega_2$ & Graphics \\
\hline
 $u_1$    &  61.9634 & ~$\{0\}$         & $\{x_1>0,x_2>0\}$ & ~$\varnothing$    & Fig.~\ref{fig:Henonrect12sols}(a) \\
 $u_2$    & 120.7887 & ~$[u_1]$         & $\{x_1<0,x_2>0\}$ & ~$\varnothing$    & Fig.~\ref{fig:Henonrect12sols}(b) \\
 $u_3$    & 122.4078 & ~$[u_1]$         & $\{x_1<0,x_2<0\}$ & ~$\varnothing$    & Fig.~\ref{fig:Henonrect12sols}(c) \\
 $u_4$    & 126.6988 & ~$[u_1]$         & $\{x_2>0\}$       & ~$\varnothing$    & Fig.~\ref{fig:Henonrect12sols}(d) \\
 $u_5$    & 125.3561 & ~$[u_1]$         & $\{x_1>0,x_2>0\}$ & $\{x_1<0,x_2<0\}$ & Fig.~\ref{fig:Henonrect12sols}(e) \\
 $u_6$    & 177.6068 & ~$[u_1,u_2]$     & $\{x_1<0,x_2<0\}$ & ~$\varnothing$    & Fig.~\ref{fig:Henonrect12sols}(f) \\
 $u_7$    & 187.1379 & ~$[u_1,u_3]$     & $\{x_2>0\}$       & ~$\varnothing$    & Fig.~\ref{fig:Henonrect12sols}(g) \\
 $u_8$    & 189.9406 & ~$[u_1,u_4]$     & $\{x_1<0,x_2<0\}$ & ~$\varnothing$    & Fig.~\ref{fig:Henonrect12sols}(h) \\
 $u_9$    & 230.0141 & ~$[u_1,u_2,u_6]$ & $\{x_1>0,x_2<0\}$ & ~$\varnothing$    & Fig.~\ref{fig:Henonrect12sols}(i) \\
 $u_{10}$ & 247.0220 & ~$[u_1,u_2,u_6]$ & $\{x_2<0\}$       & $\{x_2>0\}$       & Fig.~\ref{fig:Henonrect12sols}(j) \\
 $u_{11}$ & 250.6746 & ~$[u_1,u_2,u_6]$ & $\{x_1x_2>0\}$    & $\{x_1x_2<0\}$    & Fig.~\ref{fig:Henonrect12sols}(k) \\
 $u_{12}$ & 255.9728 & ~$[u_1,u_2,u_6]$ & $\{x_1x_2<0\}$    & ~$\varnothing$    & Fig.~\ref{fig:Henonrect12sols}(l) \\
\hline
\end{tabular}
\end{table}

\begin{figure}[!t]
\centering
\includegraphics[width=.22\textwidth]{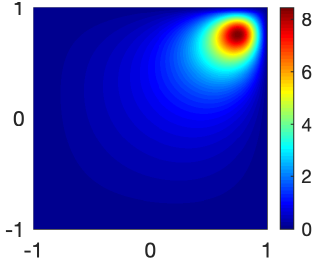}
\includegraphics[width=.22\textwidth]{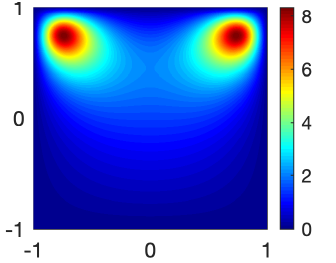}
\includegraphics[width=.22\textwidth]{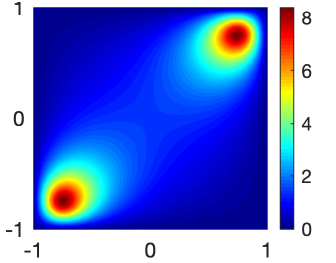}
\includegraphics[width=.22\textwidth]{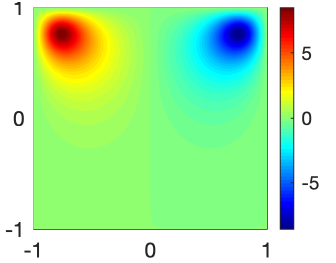} \\
\hspace*{.12\textwidth} (a)\hfill (b)\hfill (c)\hfill (d) \hspace*{.14\textwidth} \\[5pt]
\includegraphics[width=.22\textwidth]{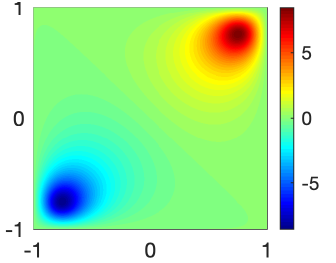}
\includegraphics[width=.22\textwidth]{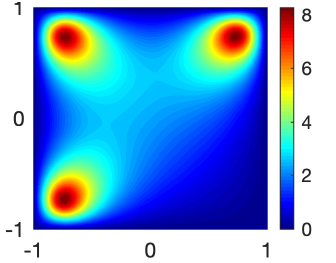}
\includegraphics[width=.22\textwidth]{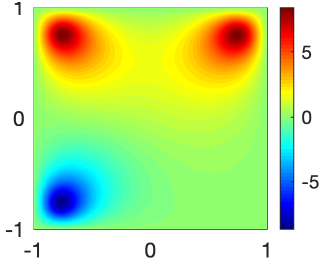}
\includegraphics[width=.22\textwidth]{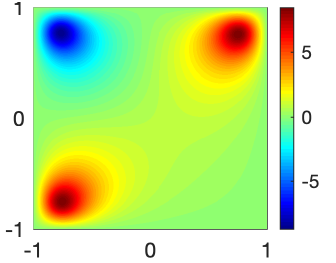} \\
\hspace*{.12\textwidth} (e)\hfill (f)\hfill (g)\hfill (h) \hspace*{.14\textwidth} \\[5pt]
\includegraphics[width=.22\textwidth]{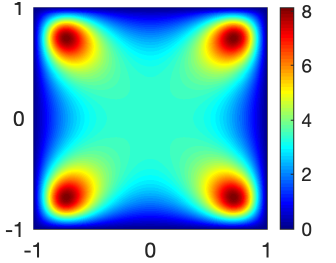}
\includegraphics[width=.22\textwidth]{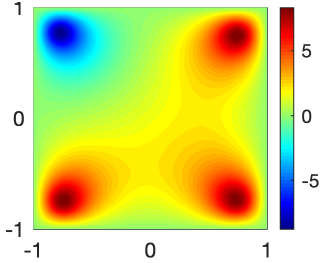}
\includegraphics[width=.22\textwidth]{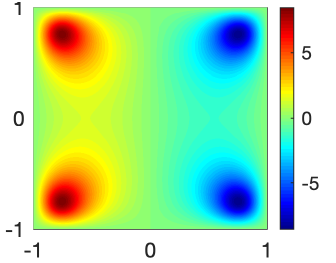}
\includegraphics[width=.22\textwidth]{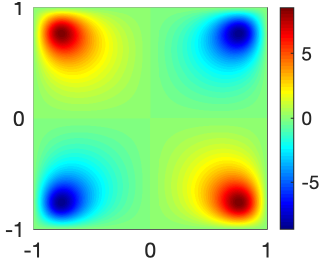} \\
\hspace*{.12\textwidth} (i)\hfill (j)\hfill (k)\hfill (l) \hspace*{.14\textwidth} \\
\caption{Profiles of twelve different solutions of the H\'{e}non equation:
(a) the single-peak positive (ground state) solution $u_1$ concentrated mainly on the corner;
(b) a two-peak positive solution $u_2$ concentrated mainly on two adjacent corners;
(c) a two-peak positive solution $u_3$ concentrated mainly on two diagonal corners;
(d) a two-peak sign-changing solution $u_4$ concentrated mainly on two adjacent corners;
(e) a two-peak sign-changing solution $u_5$ concentrated mainly on two diagonal corners;
(f) a three-peak positive solution $u_6$;
(g)-(h) two three-peak sign-changing solutions $u_7$ and $u_8$;
(i) a four-peak positive solution $u_9$; 
(j)-(l) three four-peak sign-changing solutions $u_{10}$, $u_{11}$ and $u_{12}$.}
\label{fig:Henonrect12sols}
\end{figure}

\begin{figure}[!t]
\centering
\includegraphics[width=.75\textwidth]{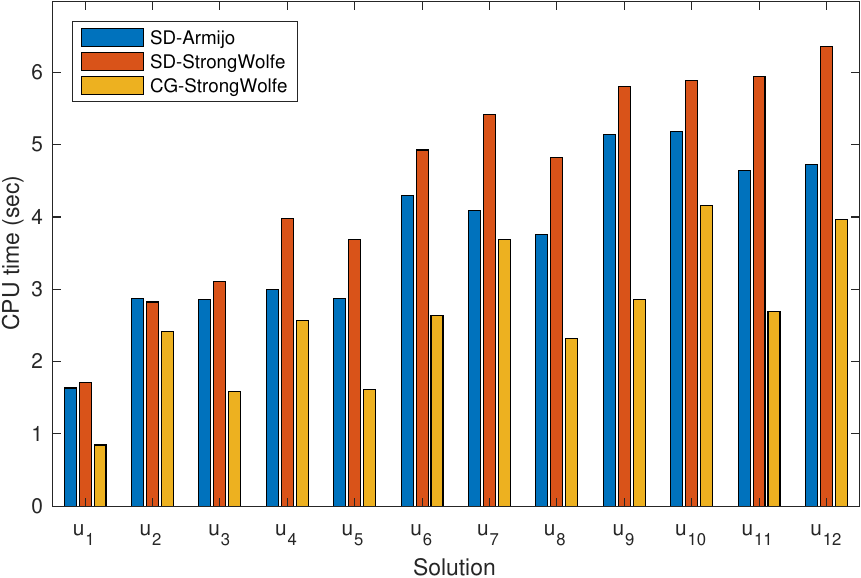}
\caption{Comparison of CPU time of LMMs for finding solutions of the H\'{e}non equation in Fig. \ref{fig:Henonrect12sols}.}
\label{fig:Henoncpubar}
\end{figure}

\subsection{Numerical results for the Chandrasekhar equation}\label{sec:numer-Chand}

Here, we report numerical results of Case~3, i.e., the Chandrasekhar equation as 
\begin{equation}
\left\{
\begin{aligned}
  -\Delta u(\mathbf{x}) &=(u^2(\mathbf{x})+2u(\mathbf{x}))^{3/2}, & \mathbf{x}\in\Omega, \\
  u(\mathbf{x}) &=0, & \mathbf{x}\in\partial\Omega.
\end{aligned}
\right.
\end{equation}
We consider a dumbbell-shaped domain $\Omega$ as depicted in Fig.~\ref{fig:CdrskDumb7sols}(a). It contains a smaller disk centered at $(-1,0)$ with radius $0.5$ and a larger disk centered at $(2,0)$ with radius $1$. A corridor of width $0.4$, symmetric respect to the $x_1$-axis, is constructed to link the two disks. In this case, we focus on finding multiple positive solutions. Limited by the length of the paper, we only present seven different positive solutions, labeled by $u_1,u_2,\ldots,u_{7}$, in Fig.~\ref{fig:CdrskDumb7sols}(b)-(h) for their profiles and features. The corresponding information on the support space $L$, subdomains $\Omega_1$ and $\Omega_2$ used in \eqref{eq:poisson-v0}, and energy values of these solutions is listed in Table~\ref{tab:CdrskDumb}. In addition, numerical comparisons of the {\tt SD-StrongWolfe}, {\tt CG-StrongWolfe} and {\tt SD-Armijo} in terms of CPU times for computing these solutions of the Chandrasekhar equation are provided in Fig.~\ref{fig:Chandrasekharcpubar}. Finally, the comparison of the computational efficiency of the three LMMs by increasing the elements and then freedoms in the FEM for computing the gradient direction $g_k$ in the iterations for finding $u_1$ are listed in Table \ref{tab:CdrskDumb2sols}.

From Figs.~\ref{fig:CdrskDumb7sols}-\ref{fig:Chandrasekharcpubar}, Tables~\ref{tab:CdrskDumb}-\ref{tab:CdrskDumb2sols} and additional results not shown here, we observe that three LMM algorithms considered can effectively find multiple positive solutions of the Chandrasekhar equation. It is also observed that the FEM mesh size has little effect on the number of iterations of the three algorithms. Again, the {\tt CG-StrongWolfe} also shows the best performance in this case.

Above all, numerical experiments in this section indicate that the CG-type direction indeed speed up the LMM greatly.

\begin{table}[!t]
\centering
\caption{The information to corresponding solutions of the Chandrasekhar equation in Fig.~\ref{fig:CdrskDumb7sols}.}
\label{tab:CdrskDumb}
\small
\begin{tabular}{lrllc}
\hline
$u_n$ & ~~~$E(u_n)$~ & ~~$L$~~~~~ & ~$\Omega_1$ ($\Omega_2=\varnothing$) & Graphics \\
\hline
 $u_1$ &   1.6624 & ~$\{0\}$     & $\{(x_1-2)^2+x_2^2<1\}$      & Fig.~\ref{fig:CdrskDumb7sols}(b) \\
 $u_2$ &  18.0067 & ~$\{0\}$     & $\{(x_1+1)^2+x_2^2<0.5\}$    & Fig.~\ref{fig:CdrskDumb7sols}(c) \\
 $u_3$ & 108.0580 & ~$\{0\}$     & $\{(x_1-0.25)^2+x_2^2<0.1\}$ & Fig.~\ref{fig:CdrskDumb7sols}(d) \\
 $u_4$ &  19.6691 & ~$[u_1]$     & $\{(x_1+1)^2+x_2^2<0.5\}$    & Fig.~\ref{fig:CdrskDumb7sols}(e) \\
 $u_5$ & 109.6897 & ~$[u_1]$     & $\{(x_1-0.25)^2+x_2^2<0.1\}$ & Fig.~\ref{fig:CdrskDumb7sols}(f) \\
 $u_6$ & 125.8846 & ~$[u_2]$     & $\{(x_1-0.25)^2+x_2^2<0.1\}$ & Fig.~\ref{fig:CdrskDumb7sols}(g) \\
 $u_7$ & 127.5247 & ~$[u_1,u_2]$ & $\{(x_1-0.25)^2+x_2^2<0.1\}$ & Fig.~\ref{fig:CdrskDumb7sols}(h) \\
\hline
\end{tabular}
\end{table}

\begin{figure}[!t]
\centering
\includegraphics[width=.45\textwidth]{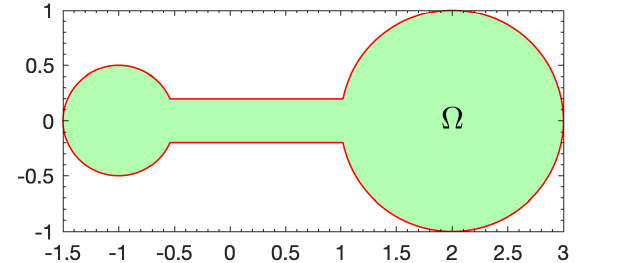}
\includegraphics[width=.45\textwidth]{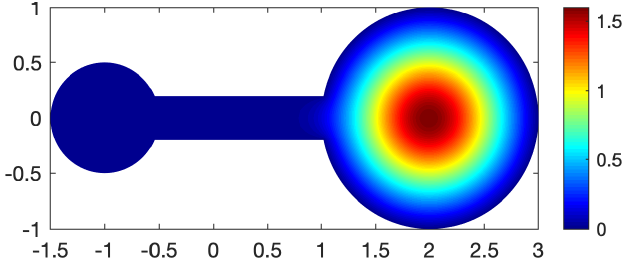} \\
\hspace*{.24\textwidth} (a)\hfill (b) \hspace*{.25\textwidth} \\[5pt]
\includegraphics[width=.45\textwidth]{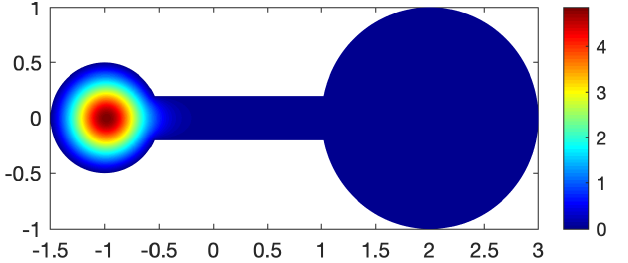}
\includegraphics[width=.45\textwidth]{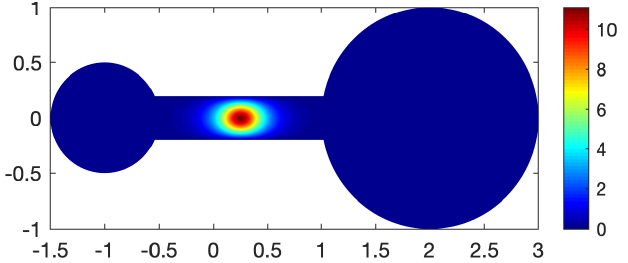} \\
\hspace*{.24\textwidth} (c)\hfill (d) \hspace*{.25\textwidth} \\[5pt]
\includegraphics[width=.45\textwidth]{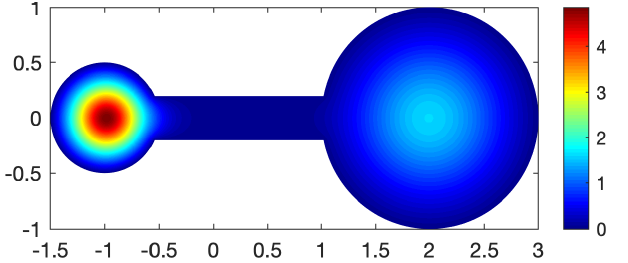}
\includegraphics[width=.45\textwidth]{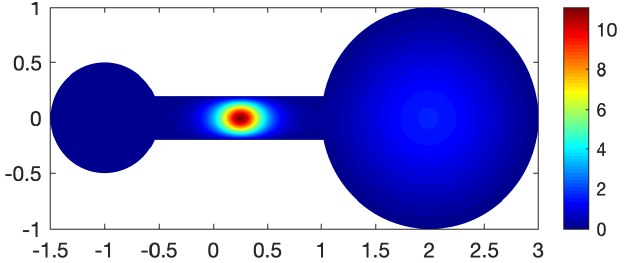} \\
\hspace*{.24\textwidth} (e)\hfill (f) \hspace*{.25\textwidth} \\[5pt]
\includegraphics[width=.45\textwidth]{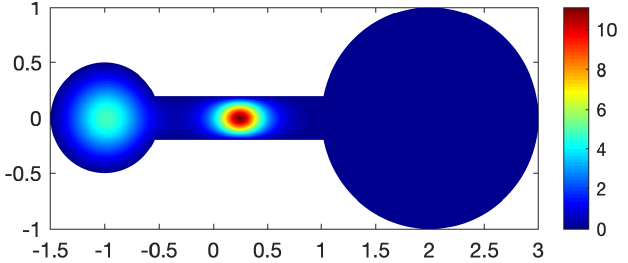} 
\includegraphics[width=.45\textwidth]{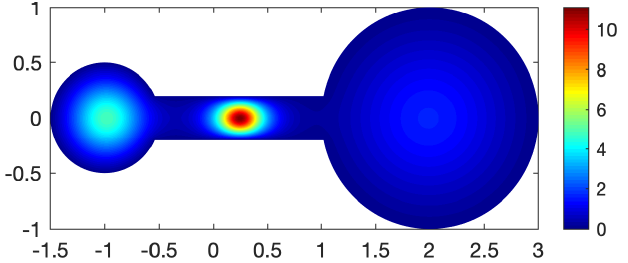} \\
\hspace*{.24\textwidth} (g)\hfill (h) \hspace*{.25\textwidth} \\
\caption{(a) A dumbbell-shaped domain. 
(b)-(h) Profiles of seven different positive solutions of the Chandrasekhar equation: 
(b) the single-peak positive ground state solution $u_1$ concentrated mainly on the larger disk;
(c) a single-peak positive solution $u_2$ concentrated mainly on the smaller disk;
(d) a single-peak positive solution $u_3$ concentrated mainly on the corridor;
(e)-(g) three two-peak positive solutions $u_4$, $u_5$ and $u_6$; 
(h) a three-peak positive solution $u_7$.}
\label{fig:CdrskDumb7sols}
\end{figure}

\begin{figure}[!t]
\centering
\includegraphics[width=.75\textwidth]{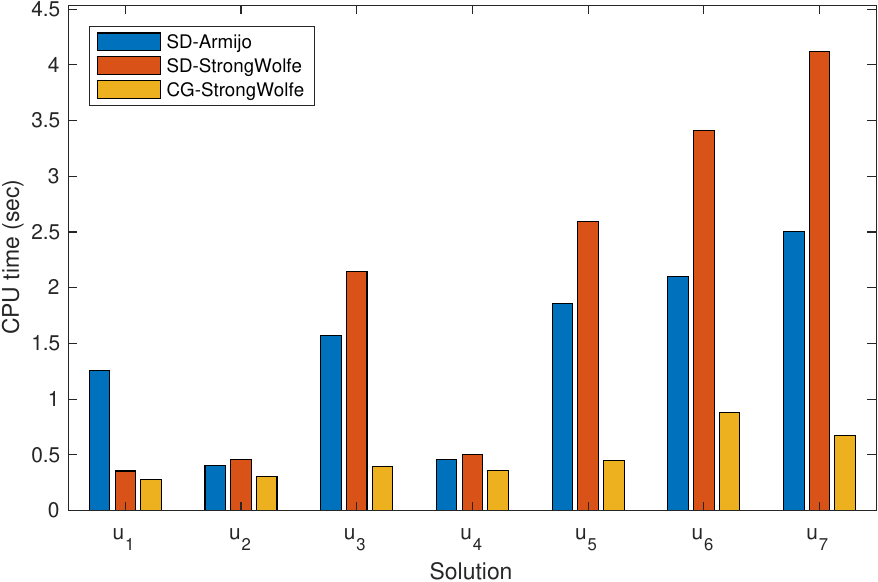}
\caption{Comparison of CPU time of LMMs for finding solutions of the Chandrasekhar equation in Fig.~\ref{fig:CdrskDumb7sols}.}
\label{fig:Chandrasekharcpubar}
\end{figure}

\begin{table}[!t]
\centering
\caption{Comparison of the number of iterations (\#its) and CPU time (time) in seconds of LMMs with different number of triangular elements ($n_T$) used in the FEM for computing the gradient direction $g_k$ in the iterations for finding the solution $u_1$ of the Chandrasekhar equation in Fig.~\ref{fig:CdrskDumb7sols}.}
\label{tab:CdrskDumb2sols}
\small
\begin{tabular}{lcccccc}
\hline
LMMs & \multicolumn{2}{c}{\tt SD-Armijo} &\multicolumn{2}{c}{\tt SD-StrongWolfe} &\multicolumn{2}{c}{\tt CG-StrongWolfe}\\
$n_T$ & \#its & time & ~\#its & time & ~\#its & time \\
\hline
3744    &51 & 0.4329   &13 & 0.1270   &10 &0.0978 \\
6050    &51 & 0.5782   &13 & 0.1704   &9  &0.1314 \\
9732    &51 & 0.8287   &13 & 0.2540   &9  &0.1757 \\
15552   &51 & 1.2944   &13 & 0.3921   &9  &0.2843 \\
55472   &51 & 4.2249   &13 & 1.2340   &9  &0.9055 \\
152696  &51 & 12.018   &13 & 3.3641   &9  &2.4026 \\
635658  &51 & 72.296   &13 & 19.660   &9  &14.258 \\
\hline
\end{tabular}
\end{table}

\section{Conclusions}\label{sec:cons}

In this paper, we introduced a framework of normalized Wolfe-Powell-type local minimax method (NWP-LMM) based on general descent directions and the normalized Wolfe-Powell-type and strong Wolfe-Powell-type step-size search rules for finding multiple unstable solutions of semilinear elliptic problems. Under certain conditions on the local peak selection and general descent directions, the feasibility and global convergence of the NWP-LMM were rigorously verified in the functional analysis level. In addition, two feasible types of descent directions, i.e., preconditioned steepest descent directions and the conjugate gradient-type direction, were proposed and discussed. The global convergence of the NWP-LMM combined with the preconditioned steepest descent directions was also provided. Extensive numerical results for several semilinear elliptic equations, including the nonlinear Schr\"{o}dinger equation, H\'{e}non equation and Chandrasekhar equation in 2D, were reported with their multiple solutions displayed to illustrate the effectiveness and robustness of our approach. The superior numerical performance of the NWP-LMM combined with the conjugate gradient-type direction was observed in extensive numerical experiments, while the rigorous verification for its global convergence is ongoing. Furthermore, designing more efficient preconditioned steepest descent or preconditioned conjugate gradient-type directions within the framework of LMM by constructing appropriate preconditioners to further improve the efficiency of computing multiple solutions will be our future work. Finally, it is worthwhile to point out that, following the line of our approach, the steepest descent direction can be replaced by a general descent direction in the devise of the traditional normalized Armijo-type and Goldstein-type local minimax algorithms, and both the feasibility and global convergence of the them can be verified.

\footnotesize
\bibliographystyle{abbrv}
\bibliography{ref_nwplmm}

\end{document}